\documentclass[11pt]{article}
\usepackage{latexsym,amsmath,amssymb,paralist,subfigure,graphicx}
\usepackage{multirow}
\usepackage{comment}
\usepackage[compress]{cite}
\usepackage{algorithm}
\usepackage{algpseudocode}
\usepackage[title]{appendix}
\usepackage{xcolor}

\usepackage{graphicx}
\usepackage{epstopdf}
\usepackage{bm}
\usepackage{hyperref}
\usepackage[normalem]{ulem}

\newcommand{\reff}[1]{{\rm (\ref{#1})}}

\newcommand{\R}{\mathbb{R}}            

\newcommand{\bc}{\pmb{c}}           
\newcommand{\bb}{\pmb{b}}            
\newcommand{\bn}{\mathbf n}            
\newcommand{\bu}{\pmb{u}}             
\newcommand{\bv}{\pmb {v}}            

\newcommand{\cJ}{\mathcal J}

\newtheorem{theorem}{Theorem}[section]
\newtheorem{lemma}[theorem]{Lemma}

\newenvironment{proof}[1][Proof]{\begin{trivlist}
\item[\hskip \labelsep {\bfseries #1}]}{\end{trivlist}}

\newcommand{\qed}{\nobreak \ifvmode \relax \else
      \ifdim\lastskip<1.5em \hskip-\lastskip
      \hskip1.5em plus0em minus0.5em \fi \nobreak
      \vrule height0.75em width0.5em depth0.25em\fi}

\def\XXint#1#2#3{{\setbox0=\hbox{$#1{#2#3}{\int}$}
\vcenter{\hbox{$#2#3$}}\kern-.51\wd0}}

{\allowdisplaybreaks

\begin{document}


\title{Structure-Preserving and Efficient Numerical Methods for Ion Transport}

\author{Jie Ding\thanks{
Department of Mathematics and Mathematical Center for Interdiscipline Research, Soochow University, 1 Shizi Street, Suzhou 215006, Jiangsu, China}
\and
Zhongming Wang\thanks{
Department of Mathematics and Statistics, Florida International University, Miami, FL, 33199, U. S. A.}
\and
Shenggao Zhou\thanks{
Corresponding author. Department of Mathematics and Mathematical Center for Interdiscipline Research, Soochow University, 1 Shizi Street, Suzhou 215006, Jiangsu, China. E-mail: sgzhou@suda.edu.cn.}
}
\maketitle
\begin{abstract}
Ion transport, often described by the Poisson--Nernst--Planck (PNP) equations, is ubiquitous in electrochemical devices and many biological processes of significance. In this work, we develop conservative, positivity-preserving, energy dissipating, and implicit finite difference schemes for solving the multi-dimensional PNP equations with multiple ionic species. A central-differencing discretization based on harmonic-mean approximations is employed for the Nernst--Planck (NP) equations. The backward Euler discretization in time is employed to derive a fully implicit nonlinear system, which is efficiently solved by a newly proposed Newton's method. The improved computational efficiency of the Newton's method originates from the usage of the electrostatic potential as the iteration variable, rather than the unknowns of the nonlinear system that involves both the potential and concentration of multiple ionic species.  Numerical analysis proves that the numerical schemes respect three desired analytical properties (conservation, positivity preserving, and energy dissipation) fully discretely. Based on advantages brought by the harmonic-mean approximations, we are able to establish estimate on the upper bound of condition numbers of coefficient matrices in linear systems that are solved iteratively. The solvability and stability of the linearized problem in the Newton's method are rigorously established as well. Numerical tests are performed to confirm the anticipated numerical accuracy, computational efficiency, and structure-preserving properties of the developed schemes. Adaptive time stepping is implemented for further efficiency improvement. Finally, the proposed numerical approaches are applied to characterize ion transport subject to a sinusoidal applied potential. \\

\end{abstract}

\section{Introduction}\label{s:Introduction}
Ion transport plays a fundamental role in many applications, such as electrochemical energy devices~\cite{BTA:PRE:04}, electrokinetics in microfluidics~\cite{Schoch_RMP08}, and transmembrane ion channels~\cite{IonChanel_HandbookCRC15}. It is often described by the so-called Poisson--Nernst--Planck (PNP) equations, which consist of the Poisson's equation and the Nernst--Planck (NP) equations.  Based on a mean-field approximation, the NP equations describe the diffusion of ions in the gradient of the electrostatic potential. The Poisson's equation determines the electrostatic potential with the charge density arising from diffusing ions.  Recently, there has been growing interests in incorporating effects that are beyond the mean-field description, e.g., the steric effect, inhomogeneous dielectric effect, and ion-ion correlations~\cite{BazantSteric_PRE07, HyonLiuBob_CMS10, BZLu_BiophyJ11, ZhouWangLi_PRE11,  HyonLiuBob_JPCB12, LiLiuXuZhou_Nonliearity13, GLin_Cicp14, LinBob_CMS14, BZLu_JCP14, NuoZhouMcCammon_JPCB14, JZWu_JPCM14, XuMaLiu_PRE14, BZLu_JSP16,  LiuJiXu_SIAP18, SWZ_CMS18}.

In this work, we develop efficient and structure-preserving  finite difference schemes for the PNP equations
\begin{equation}\label{PNPintro}
\left\{
\begin{aligned}
&\partial_t c^l= \nabla\cdot(\nabla c^l+q^lc^l\nabla\psi),~ l=1, \cdots, M,\\
&-\kappa \Delta \psi =  \sum_{l=1} ^M q^l c^l + \rho^f.
\end{aligned}
\right.
\end{equation}
Here $c^l$ is the ion concentration for the $l$-th species, $q^l$ is the valence of the $l$-th ionic species, $\kappa >0$ is a coefficient arising from nondimensionalization, $\psi$ is the electrostatic potential, and $\rho^f$ is the fixed charge density.

The analytical solutions to (\ref{PNPintro}) with zero-flux boundary conditions possess several physically desired properties, including mass conservation, positivity preservation, and free-energy dissipation, i.e.,
\begin{subequations}\label{str}
\begin{align}
& \int_{\Omega} c^l(t,\cdot )\,dV=\int_{\Omega} c^l_{\rm in}(\cdot )\,dV, \quad \forall t>0,\\
& c^l_{\rm in}(\cdot )  > 0  \Longrightarrow  c^l(t, \cdot ) > 0,  \quad \forall t>0, \\
&  \frac{d}{dt}  F = -\sum_{l=1}^M \int_{\Omega} \frac{1}{c^l}(|\nabla c^l +q^lc^l \nabla \psi |^2 )dV \leq 0, \label{dF/dt}
\end{align}
\end{subequations}
where the free energy $ F$, aside from some boundary contributions, is defined by
\begin{equation}\label{FNoBC}
 F = \sum_{l=1}^M  \int_{\Omega} \left[c^l\log c^l+\frac{1}{2} (q^lc^l+\rho^f)\psi\right]dV.
\end{equation}
The free energy contains both an entropic contribution and an electrostatic energy: $c^l {\rm log} c^l$ is the entropy related to the Brownian motion of each ion species,  and $\frac{1}{2}( q^lc^l+\rho^f)\psi$ is the (mean-field) electrostatic energy of the Coulomb interaction between charged ions.  The concentrations are expected to converge to an equilibrium solution in a closed system regardless of how initial data are distributed.

These nice mathematical features are crucial for the analytical study of the PNP equations.  For instance, 
by an energy estimate with the control of the free-energy dissipation, the solution is shown to converge to the thermal equilibrium state as time becomes large, if the boundary conditions are in thermal equilibrium (see, e.g.,~\cite{GG96}).  Long time behavior was studied in~\cite{BHN94},  and further in~\cite{AMT00, BD00} with refined convergence rates.  Results for the drift-diffusion model, i.e., the PNP equations in the semiconductor literature, with regarding global existence, uniqueness, and asymptotic behavior in the case of different boundary conditions have been established in the works~\cite{Mock_SIMA74, GG86, FI95a, FI95b}.

The PNP equations can hardly be solved analytically due to the nonlinear coupling of the electrostatic potential and ionic concentrations. Much effort has been devoted to the development of numerical methods in various applications~\cite{CCK00, ProhlSchmuck09, LHMZ10, ZCW11, AMEKLL14, CCAO14, LW14, Gibou_JCP14, LiuShu_SCM16, SunSunZhengLin16,  MXL16, AFJKXL17, GaoHe_JSC17, LW17, DSWZhou_CICP18, SWZ_CMS18, QianWangZhou_JCP2019, DingWangZhou_NMTMA19}. The existing algorithms range from finite difference to finite elements in both one dimension and high dimensions. Among these developed schemes, several attempts are made to design desirable numerical schemes that respect the nice properties \eqref{str} of analytical solutions.  In~\cite{AMEKLL14}, for instance, a second-order conservative, energy dissipative finite difference method was presented for the one-dimensional PNP equations.  A delicate temporal discretization scheme was designed to preserve energy dynamics in~\cite{AFJKXL17}. A hybrid conservative scheme that uses adaptive grids was developed to solve the PNP equations on irregular domains~\cite{Gibou_JCP14}. A type of finite difference schemes has been developed using the Slotboom transformation of the Nernst--Planck (NP) equations:
\begin{equation}\label{NPEs}
\partial_t c^l =\nabla \cdot\left(e^{-q^l\psi}\nabla g^l\right),
\end{equation}
where $g^l=c^le^{q^l\psi}$ are the Slotboom variables~\cite{PMarkowich_Book, LHMZ10, GLin_Cicp14, LW17, LiuJiXu_SIAP18}. By using the Slotboom variables, Liu and Wang~\cite{LW14} developed a free energy satisfying finite difference scheme that preserves those three properties with rigorous proof in one dimension. A free energy satisfying discontinuous Galerkin method was also developed by the same authors in~\cite{LW17}, in which the positivity of numerical solutions was not proved but enforced by an accuracy-preserving limiter. An implicit finite difference scheme was developed  to solve the PNP equations with properties of positivity preservation and energy dissipation~\cite{HuHuang_Sub2019}. The resulting nonlinear discretization system was numerically solved by a fixed-point iteration method.  Gao and He~\cite{GaoHe_JSC17} proposed a linearized convergent finite element scheme that conserves total concentration and preserves the electric energy. Rigorous error analysis of finite element type methods for the PNP equations has been studied in~\cite{SunSunZhengLin16, GaoSun_JSC18}.
A finite element discretization that can enforce positivity of numerical solutions was developed for the PNP equations, as well as the PNP equations coupling with the incompressible Navier-Stokes equations~\cite{MXL16}.

Although some progress has been made on the development of numerical methods that ensure the desired properties, it is still desirable to have computationally efficient and robust finite difference schemes  that incorporate all three desired properties together, especially in high dimensions. Among three properties \eqref{str}, the preservation of positivity is crucial to the validity of a numerical solution and is in particular hard to achieve. For instance, the positivity has been proved in~\cite{AMEKLL14} for the one-dimensional PNP equations, under assumptions that the gradient of the electrostatic potential is bounded and mesh step sizes satisfy certain constraint conditions. The numerical scheme developed in~\cite{LW14} has been proved to preserve positivity of the numerical solution but in one-dimensional case. Furthermore, one constraint on a mesh ratio needs to be satisfied to ensure positivity, due to the explicit nature of the scheme. From a practical point of view,  implicit or semi-implicit schemes that preserve positivity are much more computationally efficient, because larger time-stepping sizes are allowed in temporal integration. In our recent work~\cite{DingWangZhou2019}, we proposed efficient semi-implicit schemes that respect the desired properties. While the scheme allows relatively large time steps and is successful in positivity persevering and mass conservation, the energy dissipation is only proved in a semi-discrete form.

In this work, we develop implicit finite difference schemes for the multi-dimensional PNP equations with multiple ionic species that respect conservation, positivity preserving, and energy dissipation at fully discrete level. The NP equations reformulated in the Slotboom variables are spatially discretized by a central-differencing scheme based on harmonic-mean approximations. The backward Euler method in time is employed to derive a nonlinear coupled system, which is efficiently solved by a newly proposed Newton's method.  The improved efficiency of the Newton's method is achieved via using the electrostatic potential as the iteration variable, rather than the unknowns of the nonlinear system that involves both the potential and concentrations of multiple ionic species.  The advantages of the proposed Newton's method in saving memory and computational efficiency become more significant when the number of ionic species gets larger. In addition, numerical simulations demonstrate that the Newton's method requires appreciably fewer iteration steps and less computational time, in comparison with a typical fixed-point iteration method.  Such a Netwon's method can be employed to solve nonlinear systems resulting from other implicit discretization of the PNP-type equations.

We perform detailed numerical analysis to prove that the numerical schemes respect three desired analytical properties fully discretely. Thanks to the advantages brought by the harmonic-mean approximations, we are able to establish upper bounds on condition numbers of coefficient matrices of linear systems resulting from both the discretization of the NP equations, and solvability and stability of the linearized problem in the Newton’s method. The linear systems are efficiently solved using iterative methods with preconditioners. Numerical simulations are presented to demonstrate that the developed schemes have expected numerical accuracy, computational efficiency, and structure-preserving properties. An adaptive time stepping strategy is employed to achieve further improvement in computational efficiency. The benefit of the adaptive time stepping is demonstrated in the application of the proposed numerical approaches to probing ion transport in response to an alternating applied potential. Finally, the developed numerical approaches are applied to understand charge dynamics in electrolytes between two parallel electrodes exposed to sinusoidal applied potentials. The impact of frequency of the sinusoidal applied potentials is extensively investigated in numerical simulations. 

The rest of this paper is organized as follows. In Section~\S\ref{s:NumMethod}, we start with details of our settings, and present our implicit finite difference method in both spatial and temporal discretization. In Section~\S\ref{s:Properties} we prove the main properties at fully discrete level, including conservation, free-energy dissipation, positivity preservation. Section~\S\ref{s:Numerics} is devoted to numerical examples, including accuracy test, charge dynamics, and adaptive time stepping. Finally, we conclude in Section~\S\ref{s:Conclusions}.

\section{The PNP equations}\label{s:PNPE}
We consider the initial-boundary value problem
\begin{equation}\label{PNP}
\left\{
\begin{aligned}
&\partial_t c^l= \nabla\cdot(\nabla c^l+q^lc^l\nabla\psi) \mbox{~~for~} t>0  \mbox{~and~} l=1, \cdots, M,\\
&-\kappa \Delta \psi =  \sum_{l=1} ^M q^l c^l + \rho^f,  \\
&c^l(0,\cdot)=c^l_{\rm in}(\cdot), \\
&\frac{\partial c^l}{\partial  \textbf{n}}+q^lc^l \frac{\partial \psi}{\partial  \textbf{n}} =0  \mbox{~~on~}  \partial\Omega,\\
& \psi(\cdot)=V(\cdot) \mbox{~~on~} \Gamma_D,  \mbox{~~and~}\kappa \frac{\partial \psi}{\partial  \textbf{n}}  =\sigma(\cdot) \mbox{~~on~} \Gamma_N.
\end{aligned}
\right.
\end{equation}
Here $\Omega$ is a bounded domain, $\textbf{n}$ is a unit exterior normal vector on the boundary $\partial \Omega$, $c^l_{\rm in}$ are initial concentration distributions. To be general, we here consider both Dirichlet and Neumann boundary conditions for the electrostatic potential, i.e., $V(x)$ is a given electrostatic potential defined on the Dirichlet boundary $\Gamma_D$, and $\sigma(x)$ is the surface charge density defined on the Neumann boundary $\Gamma_N$ with  $\Gamma_D\cap\Gamma_N=\O$ and $\Gamma_D\cup\Gamma_N=\partial\Omega$.  The corresponding total free energy with boundary contributions is given by~\cite{LiuQiaoLu_SIAP18}
\begin{equation}\label{F}
 F = \sum_{l=1}^M  \int_{\Omega} \left[c^l\log c^l+\frac{1}{2} (q^lc^l+\rho^f)\psi\right]dV-\frac{1}{2}\int_{\Gamma_D}\kappa\frac{\partial\psi}{\partial {\bf n}}V dS +\frac{1}{2}\int_{\Gamma_{N}}\sigma\psi dS.
\end{equation}
With initial-boundary conditions given in \reff{PNP}, the property of free-energy dissipation \reff{dF/dt} can be derived as well.
\begin{theorem}
The solution to the PNP equations \reff{PNP} satisfies the energy dissipation law
\begin{equation}\label{dF/dtLaw}
\begin{aligned}
\frac{dF}{dt}= -\sum_{l=1}^M \int_{\Omega} \frac{1}{c^l}|\nabla c^l+ q^lc^l\nabla\psi|^2dV +\int_{\Gamma_N}\frac{d\sigma}{dt}\psi dS- \int_{\Gamma_D}\kappa\frac{\partial\psi}{\partial {\bf n}}\frac{dV}{dt}dS.
\end{aligned}
\end{equation}
\end{theorem}
\begin{proof}
Taking a derivative with respect to time, we have by integration by parts that
\[
\begin{aligned}
\frac{d}{dt}F=&- \int_{\Omega}\sum^{M}_{l=1} \frac{1}{c^l}|\nabla c+ q^lc^l\nabla\psi|^2+\frac{1}{2} \left((q^lc^l+\rho^f)\frac{d\psi}{dt}-q^l\frac{dc^l}{dt}\psi   \right)\,dV \\
&+\frac{1}{2}\int_{\Gamma_N}\left[\frac{d\sigma}{dt}\psi+\sigma\frac{d\psi}{dt}\right]\,dS-\frac{1}{2}\int_{\Gamma_D}\kappa\left[\frac{d}{dt}\left(\frac{\partial\psi}{\partial \bn}\right)V+\frac{\partial\psi}{\partial \bn}\frac{dV}{dt}\right]\,dS.
\end{aligned}
\]
It follows from the Poisson's equation that
\[
\begin{aligned}
&\sum_{l=1}^M\frac{1}{2}\int_{\Omega} \left((q^lc^l+\rho^f)\frac{d\psi}{dt}-q^l\frac{dc^l}{dt}\psi   \right) \,dV\\
&=\frac{1}{2}\int_{\Gamma_N}\left[\psi\frac{d\sigma}{dt}-\sigma\frac{d\psi}{dt}\right]\,dS
+\frac{1}{2}\int_{\Gamma_D}\kappa\left[\frac{d}{dt}\left(\frac{\partial\psi}{\partial \bn}\right)V-\frac{\partial\psi}{\partial \bn}\frac{dV}{dt}\right]\,dS
\end{aligned}
\]
Substituting into the above equation completes the proof. \qed
\end{proof}
One can observe from \reff{dF/dtLaw} that the total free energy of the system is dissipating if the boundary data are independent of time.
\section{Numerical Method}\label{s:NumMethod}
\subsection{Discretization and Notations}\label{ss:Discretization}
The spatial discretization is similar to our previous work~\cite{DingWangZhou2019}. For completeness of the presentation, we briefly introduce notations and recall the discretization scheme. For simplicity, we consider a 2D rectangular computational domain $\Omega=[a, b]\times[c,d]$ with boundaries
\[
\Gamma_D=\left\{(x, y): x=a ~\text{or}~ b, ~ c\leq y\leq d \right\}~ \text{and } \Gamma_N=\left\{(x, y): y=c ~\text{or}~ d, ~a\leq x\leq b \right\}.
\]
The computational domain is covered by the non-uniform grid points
$\left\{x_i, y_j\right\}$
with
\[
\begin{aligned}
a=x_{\frac{1}{2}} < x_{\frac{3}{2}} < \dots <   x_{N_x+\frac{1}{2}} = b,\\
c=y_{\frac{1}{2}} < y_{\frac{3}{2}} < \dots <   y_{N_y+\frac{1}{2}} = d,\\
\end{aligned}
\]
where $N_x$ and $N_y$ are the number of grid points along each dimension.  We also introduce grid points with integer indices:
\[
x_i =\frac{x_{i-\frac{1}{2}}+x_{i+\frac{1}{2}}}{2}\quad \text{and} \quad y_j =\frac{y_{j-\frac{1}{2}}+y_{j+\frac{1}{2}}}{2} ~~~~ \text{for } i= 1, \dots, N_x,~ j= 1, \dots, N_y.
\]
The grid spacings are given by
\[
h^x_i=x_{i+\frac{1}{2}}-x_{i-\frac{1}{2}},  ~h^y_j=y_{j+\frac{1}{2}}-y_{j-\frac{1}{2}} ~ \text{for } i= 1, \dots, N_x,~ j= 1, \dots, N_y,
\]
and
\[
h^x_{i+\frac{1}{2}}=x_{i+1}-x_{i}, ~ h^y_{j+\frac{1}{2}}=y_{j+1}-y_{j} ~ \text{for } i= 1, \dots, N_x-1,~ j= 1, \dots, N_y-1.
\]
We denote by $c^{l}_{i,j}$, $g^{l}_{i,j}$, and $\psi_{i,j}$ the semi-discrete approximations of $c^l(t, x_i,y_j)$, $g^l(t, x_i,y_j)$, and $\psi(t, x_i,y_j)$, respectively. Define discrete operators
\begin{equation*}
\begin{aligned}
&D_x^{+} f_{i,j}=\frac{f_{i+1,j}-f_{i,j}}{h^x_{i+\frac{1}{2}}},  \quad D^2_xf_{i,j}=\frac{D_x^{+} f_{i,j}- D_x^{+}f_{i-1,j}}{h^x_i}.
\end{aligned}
\end{equation*}
Discrete operators $D_y^{+}$ and $D^2_y$ can be analogously defined.
Also, we introduce
\[
\begin{aligned}
h^x_m=\min(h^x_1,\cdots,h^x_{N_x}),~h^x_M=\max(h^x_1,\cdots,h^x_{N_x}),\\
h^y_m=\min(h^y_1,\cdots,h^y_{N_y}),~h^y_M=\max(h^y_1,\cdots,h^y_{N_y}).
\end{aligned}
\]

We now present our implicit finite difference method by discretiztion in space and time separately. Note that, although our discretization is presented for a 2D case, it can be readily extended to three dimensions in a dimension-by-dimension manner.

\subsection{Finite difference method in space}
\subsubsection*{1) Spatial discretization of the Poisson's equation}
With given semi-discrete approximations $c^{l}_{i,j}$, we discretize the Poisson's equation with a central differencing stencil
\begin{equation}\label{DPssn}
-\kappa(D^2_x+D^2_y)\psi_{i,j}=\sum_{l=1}^Mq^lc^{l}_{i,j}+\rho^f_{i,j}, ~ i=1, \dots, N_x,~ j= 1, \dots, N_y.
\end{equation}
We employ central differencing stencils again to discretize the Dirichlet boundary conditions on $\Gamma_D$ by
\begin{equation}\label{DBC1}
\begin{aligned}
\frac{\psi_{0,j}+\psi_{1,j} }{2}=\psi_D(t,a, y_j),\quad
\frac{ \psi_{N_x+1,j}+\psi_{N_x,j}}{2}=\psi_D(t,b, y_j)  ~~\mbox{for}~~ j= 1, \dots, N_y,
\end{aligned}
\end{equation}
and the Neumann boundary conditions on $\Gamma_N$ by
\begin{equation}\label{NBC1}
\begin{aligned}
- \kappa D_y^{+} \psi_{i,0} &=\sigma({t,x_i,a}),\quad
 \kappa D_y^{+} \psi_{i,N_y}=\sigma({t,x_i,b})  ~~\mbox{for}~~ i= 1, \dots, N_x.\\
\end{aligned}
\end{equation}
Notice that the definition of the boundary data $\psi_D$ and $\sigma$ have been extended to the whole computational domain. In numerical implementation, the ghost points outside $\Omega$ are eliminated by coupling the discretization scheme \reff{DPssn} and boundary discretization~\reff{DBC1} and \reff{NBC1}.  The coupled difference equations can be written in a matrix form
\begin{equation}\label{DrePoisson}
\mathcal{L}\pmb{\psi}=\sum_{l=1}^M q^l \bc^{l}+\pmb{\rho}^f+\bb.
\end{equation}
Here $\mathcal{L}$ is the coefficient matrix, the column vector $\bb$ results from the boundary conditions \reff{DBC1} and \reff{NBC1}, and
$\pmb{\psi}$, $\bc^l$, and $\pmb{\rho}^f$ are vectors with components being $\psi_{i,j}$, $c^l_{i,j}$, and $\rho^f_{i,j}$, respectively.
\subsubsection*{2) Spatial discretization of the NP equations} We first introduce a Slotboom reformulation of the Nernst--Planck equations:
\begin{equation}\label{NPEs}
\partial_t c^l =\nabla \cdot\left(e^{-S^l}\nabla g^l\right),
\end{equation}
where $S^l=q^l\psi$ and $g^l=c^le^{S^l}$, which are the Slotboom variables~\cite{LHMZ10, LW17}.

It follows from central-differencing discretization of Eq.~\reff{NPEs} at $\{ x_i, y_j\}$ that
\begin{equation}\label{SemiNP}
\begin{aligned}
h^x_ih^y_j\frac{d}{dt}c^l_{i,j}=& h^y_j\left(e^{-S^l_{i+\frac{1}{2},j}}\widehat{g}_{x,i+\frac{1}{2},j}^l -e^{-S^l_{i-\frac{1}{2},j}}\widehat{g}_{x,i-\frac{1}{2},j}^l\right)\\
&+h^x_i\left(e^{-S^l_{i,j+\frac{1}{2}}}\widehat{g}_{y,i,j+\frac{1}{2}}^l -e^{-S^l_{i,j-\frac{1}{2}}}\widehat{g}_{y,i,j-\frac{1}{2}}^l\right)
:=Q_{i,j}(c^{l},S^l),
\end{aligned}
\end{equation}
where the flux is approximated by
\begin{equation}\label{gflux}
\widehat{g}^l_{x,i+\frac{1}{2},j}=\frac{c_{i+1,j}^{l}e^{S^l_{i+1,j}}-c_{i,j}^{l}e^{S^l_{i,j}}}{h^x_{i+\frac{1}{2}}}~ \text{and }~
\widehat{g}^l_{y,i,j+\frac{1}{2}}=\frac{c_{i,j+1}^{l}e^{S^l_{i,j+1}}-c_{i,j}^{l}e^{S^l_{i,j}}}{h^y_{j+\frac{1}{2}}}.
\end{equation}
Harmonic-mean approximations are proposed in \cite{DingWangZhou2019} to approximate the exponential terms at half-grid points:
\begin{align}\label{HM}
e^{-S^l_{i+{\frac{1}{2}},j}}=\frac{2e^{-S^l_{i+1,j}}e^{-S^l_{i,j}}}{e^{-S^l_{i+1,j}}+e^{-S^l_{i,j}}}~ \text{and }~
e^{-S^l_{i,j+{\frac{1}{2}}}}=\frac{2e^{-S^l_{i,j+1}}e^{-S^l_{i,j}}}{e^{-S^l_{i,j+1}}+e^{-S^l_{i,j}}}.
\end{align}
The zero-flux boundary conditions are discretized by
\begin{equation}\label{0FluxBC}
\begin{aligned}
&\widehat{g}^l_{x,\frac{1}{2},j}=0,\quad \widehat{g}^l_{x,N_x+\frac{1}{2},j}=0,~~ j= 1, \dots, N_y,\\
&\widehat{g}^l_{y,i,\frac{1}{2}}=0,\quad \widehat{g}^l_{y,i,N_y+\frac{1}{2}}=0,~~ i= 1, \dots, N_x.\\
\end{aligned}
\end{equation}

\subsection{Backward Euler method in time}
To obtain numerical solutions of concentrations at different time steps,  the semi-discrete scheme can be integrated with various ODE solvers implicitly.
With a nonuniform time step size $\Delta t^n$ and $t_n=t_{n-1}+\Delta t^n$, $c^{l,n}_{i,j}$, $g^{l,n}_{i,j}$, and $S^{l,n}_{i,j}$ are used to denote the numerical approximations of $c^l(t_n, x_i,y_j)$, $g^l(t_n, x_i,y_j)$, and $S^l(t_n, x_i,y_j)$, respectively.

We employ the backward Euler discretization in $c^l$:
\begin{equation}\label{ImplicitScheme}
\frac{c_{i,j}^{l,n+1}-c_{i,j}^{l,n}}{\Delta t^{n+1}}=Q_{i,j}(c^{l,n+1},S^{l,n+1}),
\end{equation}
which gives a fully implicit nonlinear system. Note that the zero-flux boundary conditions \eqref{0FluxBC} have been used in \eqref{ImplicitScheme}.  We further rewrite \reff{ImplicitScheme} in a matrix form as
\begin{equation}\label{CEqns}
\mathcal{A}^l(\pmb{\psi}^{n+1})\bc^{l,n+1}=\mathcal{P}\bc^{l,n},~ l=1,\cdots,M.
\end{equation}
Here $\mathcal{A}^l(\pmb{\psi}^{n+1})$ is a square matrix dependent on $q^l$ and $\pmb{\psi}^{n+1}$, $\mathcal{P}$ is a diagonal matrix given by $$\mathcal{P}=\text{diag}\left(h^x_1h^y_1,\cdots, h^x_1h^y_{N_y}, h^x_2h^y_1,\cdots, h^x_2h^y_{N_y},\cdots, h^x_{N_x}h^y_1,\cdots, h^x_{N_x}h^y_{N_y}\right),$$ and
$\bc^{l,n+1}$ and $\bc^{l,n}$ are column vectors with components being $c^{l,n+1}_{i,j}$ and $c^{l,n}_{i,j}$, respectively.

An adaptive time stepping strategy is often useful in speeding up simulations of charged systems that have time-dependent boundary input data. One efficient way to adjust the time step sizes has been proposed in~\cite{QiaoZhangTang2011, LiQiaoZhang2016}:
\begin{equation}\label{AdaptiveT}
\Delta t^{n+1}=\max\left(\Delta t_{\rm min},\frac{\Delta t_{\rm max}}{\sqrt{1+\alpha|F'(t)|^2}}\right),
\end{equation}
where $\alpha$ is a given constant, $F(t)$ is the free energy defined by \reff{F}. In our numerical implementation, the temporal derivative of the free energy is approximated by a difference quotient.  The time steps $\Delta t_{\rm min}$ and $\Delta t_{\rm max}$ dictate the lower and upper bounds of the adaptive time steps, respectively, i.e., $\Delta t_{\rm min}\leq \Delta t^{n+1}\leq \Delta t_{\rm max}$.

\section{Properties of numerical solutions}\label{s:Properties}
The numerical solution has several important properties, as stated in the following theorems.
\begin{theorem}\label{t:Conservation}
{\bf (Mass conservation)}The fully implicit scheme \reff{ImplicitScheme} respect conservation of concentration, in the sense that the total concentration remains constant in time, i.e.,
\begin{align}
\label{conserve1}
&\frac{d}{dt} \sum_{i=1}^{N_x}\sum_{j=1}^{N_y}c_{i,j}^l h^x_i h^y_j  =0, \\ \label{conserve2}
&\sum_{i=1}^{N_x}\sum_{j=1}^{N_y}c_{i,j}^{l,n+1}h^x_i h^y_j  =\sum_{i=1}^{N_x}\sum_{j=1}^{N_y}c_{i,j}^{l,n} h^x_i h^y_j.
\end{align}
\end{theorem}
\begin{proof}
It follows from \reff{SemiNP} that
\begin{align*}
   \frac{d}{dt} \sum_{i=1}^{N_x}\sum_{j=1}^{N_y} c^{l}_{i,j} h^x_i h^y_j &=\sum_{j=1}^{N_y}h^y_j\left(e^{-S^l_{N_x+\frac{1}{2},j}}\widehat{g}_{x,N_x+\frac{1}{2}}^l -e^{-S^l_{\frac{1}{2},j}}\widehat{g}_{x,\frac{1}{2}}^l\right)\\
   &\quad +\sum_{i=1}^{N_x}h^x_i\left(e^{-S^l_{i,N_y+\frac{1}{2}}}\widehat{g}_{y,N_y+\frac{1}{2}}^l -e^{-S^l_{i,\frac{1}{2}}}\widehat{g}_{y,\frac{1}{2}}^l\right)=0,
\end{align*}
where we have used the zero-flux boundary conditions \reff{0FluxBC} in the last step.
 Similarly, summing both sides of \reff{ImplicitScheme} over $i,j$ gives
$$
\sum_{i=1}^{N_x}\sum_{j=1}^{N_y}c_{i,j}^{l,n+1}h^x_i h^y_j -\sum_{i=1}^{N_x}\sum_{j=1}^{N_y}c_{i,j}^{l,n} h^x_i h^y_j = \Delta t^{n+1} \sum_{i=1}^{N_x}\sum_{j=1}^{N_y} Q_{i,j}(c^{l,n+1},S^{l,n+1}) h^x_i h^y_j.
$$
Incorporating  the zero-flux boundary conditions \reff{0FluxBC} at time $t_{n+1}$  for  $g^{l,n+1}=c^{l,n+1}e^{S^{n+1}}$ leads to \reff{conserve2}.  \qed
\end{proof}

\begin{theorem}\label{t:Postivity}
{\bf (Positivity preserving)}The numerical solutions $c_{i,j}^{l,n+1}$ computed from the backward Euler scheme \reff{ImplicitScheme} remain positive in time, i.e.,  if $c_{i,j}^{l,n} > 0$, then
$$
c_{i,j}^{l,n+1}> 0 ~\text{ for  } i=1,\cdots, N_x,~j=1,\cdots,N_y.
$$
\end{theorem}

\begin{proof}
Analogous to the derivation in our previous work~\cite{DingWangZhou2019}, we can show that 
\begin{equation}\label{Acol}
\left\{
\begin{aligned}
&\sum_{i=1}^{N_xN_y} \mathcal{A}^l_{i,j} =h^x_i h^y_j \quad \mbox{for}~ j=1, \dots, N_xN_y,\\
&1<\mathcal{A}^l_{i,i} <1+4\Delta t^{n+1}(\frac{h^x_M}{h^y_m}+\frac{h^y_M}{h^x_m}) \qquad\quad \mbox{for}~ i=1, \dots, N_xN_y,\\
& -2\max\{\frac{h^x_M}{h^y_m},\frac{h^y_M}{h^x_m} \}<\mathcal{A}^l_{i,j} \leq 0 \qquad\quad \mbox{for}~ i,j=1, \dots, N_xN_y,~ \mbox{and}~ i\neq j.
\end{aligned}
\right.
\end{equation}
We can verify that $\mathcal{A}^l$ is an M-matrix and $\mathcal{A}^{l,-1}>0$ in the element-wise sense. Thus, $c_{i,j}^{l,n+1}> 0$ if $c_{i,j}^{l,n}> 0$.\qed
\end{proof}

The concentrations are obtained by solving the linear system~\reff{CEqns} iteratively. For iterative methods, it is desirable to establish estimates on the condition number of the coefficient matrix.  We now recall the estimate on the condition number of the coefficient matrix $\mathcal{A}^l$ in the work~\cite{DingWangZhou2019}, and brief the corresponding proof with minor modifications.
\begin{theorem}\label{t:CondNumber}
The condition number of the coefficient matrix $\mathcal{A}^l$ satisfies
\begin{equation}\label{k1}
\kappa_1(\mathcal{A}^l):= \| \mathcal{A}^l\|_1 \| \mathcal{A}^{l,-1}\|_1 \leq \frac{h^x_M h^y_M}{h^x_m h^y_m}+8\Delta t^{n+1} \left[\frac{h^y_M}{(h^x_m)^2 h^y_m}+\frac{h^x_M}{h^x_m(h^y_m)^2}\right].
\end{equation}
\end{theorem}
\begin{proof}

From the proof of Theorem~\ref{t:Postivity}, we know that $\mathcal{A}^l$ is an M-matrix, and therefore $\mathcal{A}^{l,-1}$ exists and is a non-negative matrix. As shown in~\cite{DingWangZhou2019}, we can further prove that
$$h^x_m h^y_m\nu \mathcal{A}^{l,-1} \leq \nu,$$
where $\nu=(\underbrace{1,\cdots,1}_{N_xN_y})$. This implies that each column sum of $(\mathcal{A}^l)^{-1}$ is less or equal to $\frac{1}{h^x_m h^y_m}$. Since each element in $(\mathcal{A}^l)^{-1}$ is non-negative,  we obtain $\| (\mathcal{A}^l)^{-1}\|_1\leq\frac{1}{h^x_m h^y_m}$. Also, it follows from \reff{Acol} that
\[
\| \mathcal{A}^l\|_1  \leq h^x_M h^y_M+8\Delta t^{n+1} \left(\frac{h^y_M}{h^x_m}+\frac{h^x_M}{h^y_m}\right).
\]
This completes the proof by the definition of the $1$-norm condition number. \qed

\end{proof}

Our numerical scheme also respects the property of free-energy dissipation at full discrete level, when the boundary data is independent of time. The total discrete free energy \reff{F} is approximated by
 \begin{equation}\label{Fh}
 \begin{aligned}
 F^n_h =& \sum_{l=1}^M \sum_{i=1}^{N_x}\sum_{j=1}^{N_y} h^x_i h^y_j \left[c^{l,n}_{i,j}\log c^{l,n}_{i,j}+\frac{1}{2}(q^lc^{l,n}_{i,j}+\rho^f_{i,j})\psi^n_{i,j}\right]\\
 &-\sum_{j=1}^{N_y}h^y_j\kappa \left[V_{N_x+\frac{1}{2},j}\frac{(V_{N_x+\frac{1}{2},j}-\psi^n_{N_x,j})}{h^x_i} +V_{\frac{1}{2},j}\frac{(V_{\frac{1}{2},j}-\psi^n_{1,j})}{h^x_i}\right]\\
 &+\sum_{i=1}^{N_x}\frac{h^x_i}{2}\left[\sigma_{i,N_y+\frac{1}{2}}(h^y_j\sigma_{i,N_y+\frac{1}{2}}/\kappa+2\psi^n_{i,N_y})+\sigma_{i,\frac{1}{2}}(h^y_j\sigma_{i,\frac{1}{2}}/\kappa+2\psi^n_{i,1})\right].  \\
 \end{aligned}
 \end{equation}
Note that it is easy to verify that such an approximation is second-order accurate in space.
 \begin{theorem} \label{theorem:energy}
 {\bf (Energy dissipation)} The fully discrete free energy $F_h$ is non-increasing for time-independent boundary data in the sense that
 \begin{equation}
 \begin{aligned}\label{dFh/dt}
F_h^{n+1}-F_h^n=-\sum_{l=1}^M\sum_{i=1}^{N_x}\sum_{j=1}^{N_y} \Delta t^{n+1}\left(\frac{e^{-S^{l,n+1}_{i+\frac{1}{2},j}}}{\xi^x_i}\big|\widehat{g}^{l,n+1}_{x,i+\frac{1}{2},j}\big|^2
+ \frac{e^{-S^{l,n+1}_{i,j+\frac{1}{2}}}}{\xi^y_j}\big|\widehat{g}^{l,n+1}_{y,i,j+\frac{1}{2}}\big|^2\right)\leq 0,
\end{aligned}
\end{equation}
where $\xi^x_i$ is a number between $g^{l,n+1}_{i,j}$ and $g^{l,n+1}_{i+1,j}$, and $\xi^y_j$ is a number between $g^{l,n+1}_{i,j}$ and $g^{l,n+1}_{i,j+1}$.
\end{theorem}

\begin{proof}
The fully discrete NP equations are given by
\[
\begin{aligned}
\frac{c^{l,n+1}_{i,j}-c^{l,n}_{i,j}}{\Delta t^{n+1}}=& \frac{e^{-S^{l,n+1}_{i+\frac{1}{2},j}}\widehat{g}^{l,n+1}_{x,i+\frac{1}{2},j} -e^{-S^{l,n+1}_{i-\frac{1}{2},j}}\widehat{g}^{l,n+1}_{x,i-\frac{1}{2},j}}{h^x_i}
+ \frac{e^{-S^{l,n+1}_{i,j+\frac{1}{2}}}\widehat{g}^{l,n+1}_{y,i,j+\frac{1}{2}} -e^{-S^{l,n+1}_{i,j-\frac{1}{2}}}\widehat{g}^{l,n+1}_{y,i,j-\frac{1}{2}}}{h^y_j}.
\end{aligned}
\]
Multiplying both sides by $\log c^{l,n+1}_{i,j}+q^l \psi^{n+1}_{i,j}$ and summing over indices $i,j, l$ lead to
\begin{equation}\label{0equality}
\sum_{l=1}^M\sum_{i=1}^{N_x}\sum_{j=1}^{N_y} \frac{(c^{l,n+1}_{i,j}-c^{l,n}_{i,j})(\log c^{l,n+1}_{i,j}+q^l \psi^{n+1}_{i,j})}{\Delta t^{n+1}} +\frac{e^{-S^{l,n+1}_{i+\frac{1}{2},j}}}{\xi^x_i}\big|\widehat{g}^{l,n+1}_{x,i+\frac{1}{2},j}\big|^2 +\frac{e^{-S^{l,n+1}_{i,j+\frac{1}{2}}}}{\xi^y_j}\big|\widehat{g}^{l,n+1}_{y,i,j+\frac{1}{2}}\big|^2=0,
\end{equation}
where $\xi^x_i$ is between $c^{l,n+1}_{i,j}e^{S^{l,n+1}_{i,j}}$ and $c^{l,n+1}_{i+1,j}e^{S^{l,n+1}_{i+1,j}}$, and $\xi^y_i$ is between $c^{l,n+1}_{i,j}e^{S^{l,n+1}_{i,j}}$ and $c^{l,n+1}_{i,j+1}e^{S^{l,n+1}_{i,j+1}}$. Here the summation by parts have been used. By~\reff{0equality},  we have
\begin{equation}
\begin{aligned}
F_h^{n+1}-F_h^n=&\sum_{l=1}^M\sum_{i=1}^{N_x}\sum_{j=1}^{N_y}h^x_i h^y_j\left[c^{l,n+1}_{i,j}\log c^{l,n+1}_{i,j}-c^{l,n}_{i,j}\log c^{l,n}_{i,j}-\log c^{l,n+1}_{i,j}\left(c^{l,n+1}_{i,j}-c^{l,n}_{i,j}\right) \right]\\
&+\frac{1}{2}\sum_{l=1}^M\sum_{i=1}^{N_x}\sum_{j=1}^{N_y}h^x_i h^y_j\left[\left(q^lc^{l,n+1}_{i,j}+\rho^f_{i,j}\right)\psi^{n+1}_{i,j}-\left(q^lc^{l,n}_{i,j}+\rho^f_{i,j}\right)\psi^{n}_{i,j}-2q^l\psi^{n+1}_{i,j}\left(c^{l,n+1}_{i,j}-c^{l,n}_{i,j}\right)\right]\\
&-\sum_{j=1}^{N_y}h^y_j\kappa \left[V_{N_x+\frac{1}{2},j}\frac{(\psi^{n}_{N_x,j}-\psi^{n+1}_{N_x,j})}{h^x_i} +V_{\frac{1}{2},j}\frac{(\psi^n_{1,j}-\psi^{n+1}_{1,j})}{h^x_i}\right]\\
 &+\sum_{i=1}^{N_x}\frac{h^x_i}{2}\left[\sigma_{i,N_y+\frac{1}{2}}(2\psi^{n+1}_{i,N_y}-2\psi^n_{i,N_y})+\sigma_{i,\frac{1}{2}}(2\psi^{n+1}_{i,1}-2\psi^n_{i,1})\right] \\
&-\sum_{l=1}^M\sum_{i=1}^{N_x}\sum_{j=1}^{N_y}\Delta t^{n+1}\frac{e^{-S^{l,n+1}_{i+\frac{1}{2},j}}}{\xi^x_i}\big|\widehat{g}^{l,n+1}_{x,i+\frac{1}{2},j}\big|^2
 -\sum_{l=1}^M\sum_{i=1}^{N_x}\sum_{j=1}^{N_y}\Delta t^{n+1}\frac{e^{-S^{l,n+1}_{i,j+\frac{1}{2}}}}{\xi^y_j}\big|\widehat{g}^{l,n+1}_{y,i,j+\frac{1}{2}}\big|^2\\
&:=I_1+I_2+I_3,\\
\end{aligned}
\end{equation}
where
\begin{equation}
\begin{aligned}
I_1&=\sum_{l=1}^M\sum_{i=1}^{N_x}\sum_{j=1}^{N_y}h^x_i h^y_j\left[c^{l,n+1}_{i,j}\log c^{l,n+1}_{i,j}-c^{l,n}_{i,j}\log c^{l,n}_{i,j}-\log c^{l,n+1}_{i,j}\left(c^{l,n+1}_{i,j}-c^{l,n}_{i,j}\right) \right],\\
I_2&=\frac{1}{2}\sum_{l=1}^M\sum_{i=1}^{N_x}\sum_{j=1}^{N_y}h^x_i h^y_j\left[\left(q^lc^{l,n+1}_{i,j}+\rho^f_{i,j}\right)\psi^{n+1}_{i,j}-\left(q^lc^{l,n}_{i,j}+\rho^f_{i,j}\right)\psi^{n}_{i,j}-2q^l\psi^{n+1}_{i,j}\left(c^{l,n+1}_{i,j}-c^{l,n}_{i,j}\right)\right]\\
&\qquad-\sum_{j=1}^{N_y}h^y_j\kappa \left[V_{N_x+\frac{1}{2},j}\frac{(\psi^{n}_{N_x,j}-\psi^{n+1}_{N_x,j})}{h^x_i} +V_{\frac{1}{2},j}\frac{(\psi^n_{1,j}-\psi^{n+1}_{1,j})}{h^x_i}\right]\\
 &\qquad+\sum_{i=1}^{N_x}\frac{h^x_i}{2}\left[\sigma_{i,N_y+\frac{1}{2}}(2\psi^{n+1}_{i,N_y}-2\psi^n_{i,N_y})+\sigma_{i,\frac{1}{2}}(2\psi^{n+1}_{i,1}-2\psi^n_{i,1})\right], \\
I_3&=-\sum_{l=1}^M\sum_{i=1}^{N_x}\sum_{j=1}^{N_y}\Delta t^{n+1}\frac{e^{-S^{l,n+1}_{i+\frac{1}{2},j}}}{\xi^x_i}|\widehat{g}^{l,n+1}_{x,i+\frac{1}{2},j}|^2
 -\sum_{l=1}^M\sum_{i=1}^{N_x}\sum_{j=1}^{N_y}\Delta t^{n+1}\frac{e^{-S^{l,n+1}_{i,j+\frac{1}{2}}}}{\xi^y_j}|\widehat{g}^{l,n+1}_{y,i,j+\frac{1}{2}}|^2.\\
\end{aligned}
\end{equation}
For the term $I_1$, we have
\[
I_1=-\sum_{l=1}^M\sum_{i=1}^{N_x}\sum_{j=1}^{N_y}h^x_i h^y_j\frac{1}{2\zeta^l_{i,j}}\left(c^{l,n+1}_{i,j}-c^{l,n}_{i,j} \right)^2,
\]
where $\zeta^l_{i,j}$ is a number between $c^{l,n+1}_{i,j}$ and $c^{l,n}_{i,j}$. Here we have used the Taylor expansion to the second order and mass conservation \reff{t:Conservation}. Hence $I_1\leq 0$.

It follows from the discrete Poisson's equation \reff{DPssn} that
\[
\begin{aligned}
I_2=&\frac{1}{2}\sum_{l=1}^M\sum_{i=1}^{N_x}\sum_{j=1}^{N_y}h^x_i h^y_j\left[-\left(q^lc^{l,n+1}_{i,j}+\rho^f_{i,j}\right)\psi^{n+1}_{i,j}-\left(q^lc^{l,n}_{i,j}+\rho^f_{i,j}\right)\psi^{n}_{i,j}+2q^l\psi^{n+1}_{i,j}\left(q^lc^{l,n}_{i,j}+\rho^f_{i,j}\right) \right]\\
&\qquad-\sum_{j=1}^{N_y}h^y_j\kappa \left[V_{N_x+\frac{1}{2},j}\frac{(\psi^{n}_{N_x,j}-\psi^{n+1}_{N_x,j})}{h^x_i} +V_{\frac{1}{2},j}\frac{(\psi^n_{1,j}-\psi^{n+1}_{1,j})}{h^x_i}\right]\\
 &\qquad+\sum_{i=1}^{N_x}\frac{h^x_i}{2}\left[\sigma_{i,N_y+\frac{1}{2}}(2\psi^{n+1}_{i,N_y}-2\psi^n_{i,N_y})+\sigma_{i,\frac{1}{2}}(2\psi^{n+1}_{i,1}-2\psi^n_{i,1})\right] \\
=&\frac{1}{2}\sum_{i=1}^{N_x}\sum_{j=1}^{N_y}h^x_i h^y_j\left[\kappa \psi^{n+1}_{i,j}(D^2_x +D^2_y)\psi^{n+1}_{i,j}
    +\kappa \psi^{n}_{i,j}(D^2_x +D^2_y)\psi^{n}_{i,j}-2\kappa \psi^{n+1}_{i,j}(D^2_x +D^2_y)\psi^{n}_{i,j}\right]\\
&\qquad-\sum_{j=1}^{N_y}h^y_j\kappa \left[V_{N_x+\frac{1}{2},j}\frac{(\psi^{n}_{N_x,j}-\psi^{n+1}_{N_x,j})}{h^x_i} +V_{\frac{1}{2},j}\frac{(\psi^n_{1,j}-\psi^{n+1}_{1,j})}{h^x_i}\right]\\
 &\qquad+\sum_{i=1}^{N_x}\frac{h^x_i}{2}\left[\sigma_{i,N_y+\frac{1}{2}}(2\psi^{n+1}_{i,N_y}-2\psi^n_{i,N_y})+\sigma_{i,\frac{1}{2}}(2\psi^{n+1}_{i,1}-2\psi^n_{i,1})\right]\\
=&-\frac{1}{2}\sum_{i=1}^{N_x}\sum_{j=1}^{N_y}h^x_i h^y_j\kappa\left[ (D^+_x\psi^{n+1}_{i,j})^2+(D^+_y\psi^{n+1}_{i,j})^2+(D^+_x\psi^{n}_{i,j})^2+(D^+_y\psi^{n}_{i,j})^2\right.\\
&\qquad\left.-2D^+_x\psi^{n}_{i,j}D^+_x\psi^{n+1}_{i,j}-2D^+_y\psi^{n}_{i,j}D^+_y\psi^{n+1}_{i,j}\right]\\
=&-\frac{1}{2}\sum_{i=1}^{N_x}\sum_{j=1}^{N_y}h^x_i h^y_j \kappa \left[\left(D^+_x\psi^{n+1}_{i,j}-D^+_x\psi^{n}_{i,j}\right)^2+\left(D^+_y\psi^{n+1}_{i,j}-D^+_y\psi^{n}_{i,j}\right)^2\right]\leq 0,
\end{aligned}
\]
\end{proof}
where the summation by parts has been used in the third equality.

Clearly,  we have $I_3\leq 0$. Combining $I_1, I_2$,  and $I_3$ completes the proof. \qed

\subsection{Newton's iteration method and its viability}
The discrete PNP equations \reff{DrePoisson} and \reff{ImplicitScheme} form  a coupled  nonlinear discrete system
\begin{equation}\label{disPNP}
\left\{
\begin{aligned}
&\mathcal{L}\pmb{\psi}^{n+1}=\sum_{l=1}^M q^l \bc^{l,n+1}+\pmb{\rho}^f+\bb^{n+1},\\
&\mathcal{A}^l(\pmb{\psi}^{n+1})\bc^{l,n+1}=\mathcal{P}\bc^{l,n}, l=1,\cdots,M,\\
\end{aligned}
\right.
\end{equation}
where $\pmb{\psi}^{n+1}$ and $\bc^{l,n+1}$ are the unknowns, and $\bb^{n+1}$ is known boundary data at time $t^{n+1}$. The nonlinear system~\reff{disPNP} can be solved by iterative methods with iterative variables involving both the concentration and potential.  To save memory, one treatment is to decouple the system and use a fixed-point iterative method in which the discrete Poisson's equation and NP equations are solved alternatively. The fixed point method is simple to implement but may suffer from slow convergence.  To speed up the convergence, we further propose a novel Newton's iteration approach that uses the potential as the only iterative variables.

The electrostatic potential that solves the system~\reff{disPNP} can be found by solving the nonlinear residual equations $R(\bu)=0$, where the vector function $R: \R^{N_xN_y} \to \R^{N_xN_y}$ is defined by
\begin{equation}\label{NewtonF}
R(\bu):=\mathcal{L}\bu-\sum_{l=1}^M q^l \mathcal{A}^{l,-1}(\bu)\mathcal{P}\bc^{l,n}-\pmb{\rho}^f-\bb^{n+1},\quad   \bu \in \mathbb{R}^{N_xN_y}.
\end{equation}
We note that the matrix $\mathcal{A}^l$ is invertible as shown in the proof of Theorem~\ref{t:Postivity}. To linearize, we take the Fr\'{e}chet derivative of the residuals
\begin{equation}
\begin{aligned} \label{DerivativeF}
  \cJ (\bu)[\delta \bu]&=\frac{d R(\bu+\tau\delta \bu)}{d \tau}\bigg|_{\tau=0}\\
  &=\mathcal{L}\delta \bu-\sum_{l=1}^M q^l \frac{d\mathcal{A}^{l,-1}(\bu+\tau\delta \bu)}{d\tau}\bigg|_{\tau=0}\mathcal{P}\bc^{l,n}.\\
\end{aligned}
\end{equation}
To calculate the Fr\'{e}chet derivative of $A^{l,-1}(\bu)$, we take a derivative with respect to $\tau$ of the identity matrix
$$\mathcal{I}=\mathcal{A}(\bu+\tau\delta \bu)\mathcal{A}^{l,-1}(\bu+\tau\delta \bu).$$
We then have
\begin{equation}\label{DerivativeIdentity}
\frac{d\mathcal{A}^{l,-1}(\bu+\tau\delta \bu)}{d\tau}\bigg|_{\tau=0}=-\mathcal{A}^{l,-1}(\bu)\frac{d\mathcal{A}^l(\bu+\tau\delta \bu)}{d\tau}\bigg|_{\tau=0}\mathcal{A}^{l,-1}(\bu).
\end{equation}
Combining \reff{DerivativeF} and \reff{DerivativeIdentity}, we have
$$\cJ(\bu)[\delta \bu]=\mathcal{L}\delta \bu+\sum_{l=1}^M q^l \mathcal{A}^{l,-1}(\bu)\frac{d\mathcal{A}^l(\bu+\tau\delta \bu)}{d\tau}\bigg|_{\tau=0}\mathcal{A}^{l,-1}(\bu)\mathcal{P}\bc^{l,n}.$$
We notice that $\frac{d\mathcal{A}^l(\bu+\tau\delta \bu)}{d\tau}\bigg|_{\tau=0}\mathcal{A}^{l,-1}(\bu)\mathcal{P}\bc^{l,n}$ is a vector linearly dependent on $\delta \bu$ and can be identified as
$$\frac{d\mathcal{A}^l(\bu+\tau\delta \bu)}{d\tau}\bigg|_{\tau=0}\mathcal{A}^{l,-1}(\bu)\mathcal{P}\bc^{l,n}=q^l\mathcal{K}(\pmb{\mu}^{l},\bu)\delta \bu,$$
where $\mathcal{K}(\pmb{\mu}^{l},\bu)$ is a symmetric matrix with elements given in {\bf Appendix}~\ref{ApxA} and $\pmb{\mu}^{l}=\mathcal{A}^{l,-1}(\bu)\mathcal{P}\bc^{l,n}$. Since $\mathcal{A}^{l}$ is an M-matrix, elements in the column vector $\pmb{\mu}^{l}$ are all non-negative, given the concentration on the previous step, $\bc^{l,n}$, is non-negative.  We further denote by $\cJ(\bu)[\delta \bu]=\mathcal{W}(\bu)\delta \bu$, where
\begin{equation}
\begin{aligned}\label{NewtonW}
\mathcal{W}(\bu):= \mathcal{L} +\sum_{l=1}^M(q^l)^2\mathcal{A}^{l,-1}(\bu)\ \mathcal{K}(\pmb{\mu}^{l},\bu).
\end{aligned}
\end{equation}

In our Newton's iteration method,  for a given previous iteration step $\bu^k$, we update the potential via
$$\bu^{k+1}=\bu^k+\delta \bu,$$
where the correction $\delta \bu$ solves the linear system
\begin{equation}\label{NewtonF2}
\mathcal{W}(\bu^k)\delta \bu=-R(\bu^k).
\end{equation}
In our implementation, we solve such a linear system using the BiCGSTAB method preconditioned with an incomplete lower-upper (iLU) decomposition of the matrix $\mathcal{L}$. We remark that the iterative method only requires multiplication of the matrix and vector (i.e., $\mathcal{W}(\bu^k)\delta \bu$), rather the matrix $\mathcal{W}$ itself that involves the inverse of $\mathcal{A}^l$. Therefore, when solving the linear system~\reff{NewtonF2}, we need to solve linear systems involving the coefficient matrices $\mathcal{A}^l$, again by using the BiCGSTAB method preconditioned with iLU decompositions of $\mathcal{A}^l$.  Numerical simulations reveal that the preconditioning accelerates the convergence significantly.

We summarize the whole numerical algorithm as follows.
\begin{algorithm}[H]
\caption{ Numerical algorithm for the PNP equations}
\label{Alg}
\begin{algorithmic}[1]
\State Given initial concentrations $\bc^{l,0}$, obtain $\pmb{\psi}^0$ by solving the discrete Poisson's equation \reff{DPssn} with boundary conditions \reff{DBC1} and \reff{NBC1};
\State Given $\bc^{l,n}$ and $\pmb{\psi}^n$ at time step $t_n$. Let $k=0$ and $\bu^k=\pmb{\psi}^n$;
\State Find $\bv^{l,k}$ by solving $\mathcal{A}^l(\bu^k)\bv^{l,k}=\mathcal{P}\bc^{l,n}$;
\State Find $R(\bu^k)= \mathcal{L}\bu^k-\sum_{l=1}^M q^l  \bv^{l,k} -\pmb{\rho}^f-\bb^{n+1}$;
\State Find $\delta \bu$ by solving the linear system $\mathcal{W}(\bu^k) \delta \bu = - R(\bu^k)$ iteratively;
\State Update $\bu^{k+1}=\bu^k+\delta \bu$;
\State Check convergence. If $\|\delta \bu\|_{\infty}<\text{Tol}$, let $\pmb{\psi}^{n+1} = \bu^{k+1}$, find $\bc^{l,n+1}$ by solving $\mathcal{A}^l(\pmb{\psi}^{n+1})\bc^{l,n+1}=\mathcal{P}\bc^{l,n}$, and set $T=T+\Delta t^{n+1}$; else, let $k=k+1$ and go to Step $3$;
\State If $T \geq T_{\text{end}}$, then stop; else, let $n=n+1$ and go back to Step $2$.
\end{algorithmic}
\end{algorithm}


We now investigate the properties of the matrix $\mathcal{W}$, and solvability and stability of the corresponding linear system~\reff{NewtonF2}.



\begin{lemma} \label{lemmaL}
The coefficient matrix $\mathcal{L}$ in~\reff{DrePoisson} is an M-matrix and $\|\mathcal{L}\|_1 < \frac{4}{(h^x_m)^2}+\frac{4}{(h^y_m)^2}$.
\end{lemma}

\begin{proof}
Define
$$[i,j]:=(i-1)N_y+j~~\mbox{for}~i=1,2,\cdots,N_x,~j=1,2,\cdots,N_y.$$
Denote by $\mathcal{R}=\mathcal{PL}$. It is easy to verify the following results:
\begin{equation}\label{Rres}
\left\{
\begin{aligned}
&\sum_{m=1}^{N_xN_y} \mathcal{R}_{m,n}=0~~\mbox{for}~ n=[i,j]~\mbox{with } i= 2, \dots, N_x-1~\mbox{and}~ j= 2, \dots, N_y-1, \\
&\sum_{m=1}^{N_xN_y} \mathcal{R}_{m,n}= 2\frac{h^y_j}{h^x_{\frac{1}{2}}} ~~\mbox{for}~ n=[1,j]~\mbox{with}~ 1\leq j\leq N_y, \\
&\sum_{m=1}^{N_xN_y} \mathcal{R}_{m,n}=2\frac{h^y_j}{h^x_{N_x+\frac{1}{2}}}~~\mbox{for}~ n=[N_x,j]~\mbox{with}~ 1\leq j\leq N_y, \\
&0< \mathcal{R}_{m,m} < 2 (\frac{h^y_M}{h^x_m}+\frac{h^x_M}{h^y_m}) ~~\mbox{for}~m=1,\cdots,N_xN_y,\\
&-\max\{\frac{h^y_M}{h^x_m}, \frac{h^x_M}{h^y_m}\} < \mathcal{R}_{m,n} = \mathcal{R}_{n,m}  \leq 0~~\mbox{for}~m, n=1,\dots, N_xN_y\mbox{ and}~ m\neq n.\\
\end{aligned}
\right.
\end{equation}
For a non-zero $\pmb{\alpha}=(\alpha_1,\alpha_2,\cdots,\alpha_{N_xN_y})$, we have
\[
\begin{aligned}
\pmb{\alpha}\mathcal{R}\pmb{\alpha}^T=&2\sum_{j=1}^{N_y}\left(\frac{h^y_j}{h^x_{\frac{1}{2}}}\alpha^2_{[1,j]}+\frac{h^y_j}{h^x_{N_x+\frac{1}{2}}}\alpha^2_{[N_x,j]} \right)\\
&\quad+\sum_{i=1}^{N_x}\sum_{j=1}^{N_y-1}\frac{h^x_i}{h^y_{j+\frac{1}{2}}}\left(\alpha_{[i,j]}-\alpha_{[i,j+1]}\right)^2+\sum_{i=1}^{N_x-1}\sum_{j=1}^{N_y}\frac{h^y_j}{h^x_{i+\frac{1}{2}}}\left(\alpha_{[i,j]}-\alpha_{[i+1,j]}\right)^2>0.
\end{aligned}
\]
Thus, $\mathcal{R}$ is positive definite and eigenvalues of $\mathcal{R}$ are all positive. Moreover, $\mathcal{R}$ is an M-matrix with $\mathcal{R}^{-1}> 0$.  Thus, $\mathcal{L}^{-1}=\mathcal{R}^{-1}\mathcal{P}> 0$. Since diagonal elements of $\mathcal{L}$ are all positive and off-diagonal elements are all non-positive, we obtain that  $\mathcal{L}$ is an M-matrix as well. From \reff{Rres} and $\mathcal{L}=\mathcal{P}^{-1}\mathcal{R}$, we have $\|\mathcal{L}\|_1<\frac{4}{(h^x_m)^2}+\frac{4}{(h^y_m)^2}$.
\qed
\end{proof}


\begin{lemma}\label{InvMat}
Suppose $\mathcal{M}_1$ and $\mathcal{M}_2$ are two real square matrices. Assume that $\mathcal{M}_1$ is invertible with $\| \mathcal{M}_1^{-1}\|_1 <\infty$. If $\| \mathcal{M}_2 \|_1 < \frac{1}{ \| \mathcal{M}_1^{-1}\|_1}$, then $\mathcal{M}_1 + \mathcal{M}_2$ is invertible, and
\begin{equation}\label{InvW}
\|\left(\mathcal{M}_1 + \mathcal{M}_2\right)^{-1}\|_1\leq \frac{\| \mathcal{M}_1^{-1}\|_1}{1-\| \mathcal{M}_1^{-1}\|_1 \|\mathcal{M}_2 \|_1}.
\end{equation}
\end{lemma}

\begin{proof}
Since $\mathcal{M}_1$ is invertible,  we have
$$\mathcal{M}_1 + \mathcal{M}_2=\left(\mathcal{I}+ \mathcal{M}_2 \mathcal{M}_1^{-1} \right)\mathcal{M}_1.$$
The assumption $\| \mathcal{M}_2 \|_1 < \frac{1}{ \| \mathcal{M}_1^{-1}\|_1}$ implies that
$\| \mathcal{M}_2 \mathcal{M}_1^{-1} \|_1 < 1$.
By the matrix geometric series theorem~\cite{golub1996}, we know that  $\left(\mathcal{I}+ \mathcal{M}_2 \mathcal{M}_1^{-1}\right)^{-1}$ exists and
$$\|\left(\mathcal{I}+ \mathcal{M}_2 \mathcal{M}_1^{-1}\right)^{-1}\|_1 \leq \frac{1}{1-\| \mathcal{M}_2\|_1 \| \mathcal{M}_1^{-1} \|_1}.$$
Therefore, $\mathcal{M}_1 + \mathcal{M}_2$ is invertible with
$$\left(\mathcal{M}_1 + \mathcal{M}_2\right)^{-1}=\mathcal{M}_1^{-1}\left(\mathcal{I}+ \mathcal{M}_2 \mathcal{M}_1^{-1} \right)^{-1},$$
and its norm satisfies
$$\|\left(\mathcal{M}_1 + \mathcal{M}_2\right)^{-1}\|_1\leq \| \mathcal{M}_1^{-1} \|_1 \|\left(\mathcal{I}+ \mathcal{M}_2 \mathcal{M}_1^{-1} \right)^{-1}\|_1
\leq \frac{\| \mathcal{M}_1^{-1}\|_1}{1-\| \mathcal{M}_1^{-1}\|_1 \|\mathcal{M}_2 \|_1}. \qed$$
\end{proof}


We now derive a sufficient condition that guarantees the solvability and stability of the linearized problem~\reff{NewtonF2}. Denote by
$$(\gamma^1,\gamma^2,\cdots,\gamma^{N_x N_y}):=\underbrace{(1,1,\cdots,1)}_{N_x N_y}\mathcal{R}^{-1}.$$
Since $\mathcal{R}$ is an M-matrix (cf. \reff{Rres}), we have $\gamma^{l} >0$ for $l=1,2,\cdots,N_xN_y$. We define the index 
$$l_*:= \underset{l=1,2,\cdots,N_xN_y}{\operatorname{argmax}}\left\{\gamma^lh^x_l h^y_l\right\}.$$ 
\begin{theorem} \label{TheoremInvW}
If 
$$\Delta t^{n+1}  < \frac{1}{4 \gamma^{l_*} \left[\frac{h^y_Mh^x_{l_*}h^y_{l_*}}{(h^x_m)^2h^y_m}+\frac{h^x_Mh^x_{l_*}h^y_{l_*}}{(h^y_m)^2h^x_m}\right] \sum_{l=1}^M(q^l)^2 \|\pmb{\mu}^l\|_{\infty}},$$
then we have that\\
\qquad (1) $\mathcal{W}$ is invertible;\\
\qquad (2) $\| \mathcal{W}^{-1}\|_1\leq \frac{\gamma^{l_*}h^x_{l_*}h^y_{l_*}}{1-4\gamma^{l_*}\Delta t^{n+1}\left[\frac{h^y_Mh^x_{l_*}h^y_{l_*}}{h^y_m(h^x_m)^2}+\frac{h^x_Mh^x_{l_*}h^y_{l_*}}{h^x_m(h^y_m)^2}\right]\sum_{l=1}^M (q^l)^2\|\pmb{\mu}^l\|_{\infty}}$;\\
\qquad (3) The 1-norm condition number of $\mathcal{W}$ in \reff{NewtonF2} satisfies
\begin{equation}\label{ConNum}
\kappa_1(\mathcal{W}):= \| \mathcal{W}\|_1 \| \mathcal{W}^{-1}\|_1 \leq \frac{\gamma^{l_*}h^x_{l_*} h^y_{l_*} \left\{\frac{4}{(h^x_m)^2}+\frac{4}{(h^y_m)^2}+4\Delta t^{n+1}\left[\frac{h^y_M}{h^y_m(h^x_m)^2}+\frac{h^x_M}{h^x_m(h^y_m)^2} \right]\sum_{l=1}^M(q^l)^2 \|\pmb{\mu}^l\|_{\infty} \right\}}{1-4\gamma^{l_*}\Delta t^{n+1}\left[\frac{h^y_Mh^x_{l_*}h^y_{l_*}}{h^y_m(h^x_m)^2}+\frac{h^x_Mh^x_{l_*}h^y_{l_*}}{h^x_m(h^y_m)^2}\right]\sum_{l=1}^M (q^l)^2\|\pmb{\mu}^l\|_{\infty}}.
\end{equation}
\end{theorem}
\begin{proof}
As shown in {\bf Appendix} \ref{ApxA}, we have 
\begin{equation}\label{Bpro}
\left\{
\begin{aligned}
&\sum_{m=1}^{N_xN_y} \mathcal{K}^l_{m,n} =\sum_{n=1}^{N_xN_y} \mathcal{K}^l_{m,n} =0 \quad \mbox{for}~ m,n=1, \dots, N_xN_y,\\
&0<\mathcal{K}^l_{n,n} <2\Delta t^{n+1}(\frac{h^y_M}{h^x_m}+\frac{h^x_M}{h^y_m})\|\pmb{\mu}^l\|_{\infty} \quad \mbox{for}~ n=1, \dots, N_xN_y,\\
&-\Delta t^{n+1}\max\{\frac{h^y_M}{h^x_m},\frac{h^x_M}{h^y_m}\}\|\pmb{\mu}^l\|_{\infty}<\mathcal{K}^l_{m,n} <0 \quad \mbox{for}~ n,m=1, \dots, N_xN_y,~ \mbox{and}~ m\neq n.
\end{aligned}
\right.
\end{equation}
This implies that 
$$\|\mathcal{K}^l\|_1\leq 4\Delta t^{n+1} \left(\frac{h^y_M}{h^x_m}+\frac{h^x_M}{h^y_m}\right)\|\pmb{\mu}^l \|_{\infty}.$$
From the proof of Theorem~\ref{t:CondNumber}, we know that $\|\mathcal{A}^{l,-1}\|_1\leq \frac{1}{h^x_m h^y_m}$. Therefore, we have 
\begin{equation}
\left\|\sum_{l=1}^M(q^l)^2\mathcal{A}^{l,-1}(\bu^k)\ \mathcal{K}(\pmb{\mu}^k,\bu^k)\right\|_1
\leq 4\Delta t^{n+1} \left[\frac{h^y_M}{(h^x_m)^2h^y_m}+\frac{h^x_M}{(h^y_m)^2h^x_m}\right]\sum_{l=1}^M(q^l)^2 \|\pmb{\mu}^l \|_{\infty}.
\end{equation}
If
\[
\Delta t^{n+1}  < \frac{1}{4 \gamma^{l_*} \left[\frac{h^y_Mh^x_{l_*}h^y_{l_*}}{(h^x_m)^2h^y_m}+\frac{h^x_Mh^x_{l_*}h^y_{l_*}}{(h^y_m)^2h^x_m}\right] \sum_{l=1}^M(q^l)^2 \|\pmb{\mu}^l\|_{\infty}},
\]
we have $$\left\|\mathcal{W}-\mathcal{L} \right\|_1 < \frac{1}{\|\mathcal{L}^{-1} \|_1}.$$
It follows from Lemma \reff{InvMat} that $\mathcal{W}$ is invertible. By \reff{InvW}, we further have
$$\| \mathcal{W}^{-1}\|_1\leq \frac{\|\mathcal{L}^{-1} \|_1}{1-\|\mathcal{L}^{-1} \|_1\|\mathcal{L}-\mathcal{W} \|_1}\leq \frac{\gamma^{l_*}h^x_{l_*}h^y_{l_*}}{1-4\gamma^{l_*}\Delta t^{n+1}\left[\frac{h^y_Mh^x_{l_*}h^y_{l_*}}{h^y_m(h^x_m)^2}+\frac{h^x_Mh^x_{l_*}h^y_{l_*}}{h^x_m(h^y_m)^2}\right]\sum_{l=1}^M (q^l)^2\|\pmb{\mu}^l\|_{\infty}}.$$
By Lemma~\ref{lemmaL}, we obtain
\begin{equation}
\begin{aligned}
\|\mathcal{W} \|_1 &\leq \|\mathcal{L}\|_1+\sum_{l=1}^M(q^l)^2\|\mathcal{A}^{l,-1}(\bu^k)\|_1 \|\mathcal{K}(\pmb{\mu}^{l,k},\bu^k)\|_1\\
&\leq \frac{4}{(h^x_m)^2}+\frac{4}{(h^y_m)^2}+\sum_{l=1}^M4\Delta t^{n+1}(q^l)^2\left(\frac{h^y_M}{h^y_m(h^x_m)^2}+\frac{h^x_M}{h^x_m(h^y_m)^2} \right)\|\pmb{\mu}^l\|_{\infty}.
\end{aligned}
\end{equation}
By the definition of $1$-norm condition number, we complete the proof of \reff{ConNum}. \qed
\end{proof}


\section{Numerical tests}\label{s:Numerics}
We perform numerical simulations to show numerical accuracy of the developed numerical methods and their effectiveness in preserving mass conservation, positivity, and free-energy dissipation. The advantage of using the adaptive time stepping strategy is demonstrated through an example in which electrolytes are exposed to sudden alternating applied potentials over two electrodes. Furthermore, we apply the developed numerical methods to characterize the charge dynamics of electrolytes between two parallel electrodes with sinusoidal applied potentials.     Unless otherwise specified,  we use a uniform mesh with grid spacing $h^x_i = h^y_j$  in the following simulations. The stopping tolerance in the Newton's iterations is set to be $10^{-10}$.


\subsection{Accuracy and efficiency}
We consider an electrolyte solution with symmetric monovalent ions. To test the accuracy of our methods, we consider the following constructed problem in 2D:
\begin{equation}
\left\{
\begin{aligned}
&\partial_t c^1=\nabla\cdot(\nabla c^1 +c^1 \nabla\psi)+f_1,\\
&\partial_t c^2=\nabla\cdot(\nabla c^2 -c^2 \nabla\psi)+f_2,\\
&-\kappa\Delta \psi=c^1-c^2+\rho^f.
\end{aligned}
\right.
\end{equation}
The functions $f_1$, $f_2$, and $\rho^f$ are determined by the following exact solution
\begin{equation}
\left\{
\begin{aligned}
&c^1=\pi^2 e^{-t}\cos(\pi x)\cos(\pi y)/5+2,\\
&c^2=\pi^2 e^{-t}\cos(\pi x)\cos(\pi y)/5+2,\\
&\psi=e^{-t}\cos(\pi x)\cos(\pi y).
\end{aligned}
\right.
\end{equation}
The initial and boundary conditions are obtained by evaluating the exact solution at $t=0$ and the boundary of a  computational box, respectively.
\begin{table}[H]
\centering
\begin{tabular}{ccccccc}
\hline\hline
$h$ & $l^\infty$ error in $c^1$ & Order&$l^\infty$ error in $c^2$ & Order&$l^\infty$ error in $\psi$ & Order\\
\hline
$\frac{1}{10}$&1.30e-02 &- &1.41e-02 &- & 1.00e-03&-\\
$\frac{1}{20}$&3.40e-03 &1.93 &3.70e-03 &1.93 & 2.70e-04&1.89\\
$\frac{1}{30}$&1.50e-03 &2.01 &1.60e-03 &2.06 & 1.21e-04&1.97\\
$\frac{1}{40}$&8.45e-04 &1.99 &9.26e-04 &1.90 & 6.87e-05&1.98\\
$\frac{1}{50}$&5.42e-04 &1.99 &5.94e-04 &1.99 & 4.41e-05&1.99\\
\hline\hline
\end{tabular}
\caption{Numerical error and convergence order of numerical solutions at time $T=0.1$.}
\label{t:OrderTab}
\end{table}

We test the numerical accuracy of the proposed numerical method using various spatial step size $h$ with a fixed mesh ratio $\Delta t=h^2$. Table \ref{t:OrderTab} lists $l^\infty$ errors and convergence orders for ionic concentration and electrostatic potential at time $T=0.1$.  We observe that the error decreases as the mesh refines, and that the convergence orders for ion concentrations and the potential are both about $2$. This indicates that the fully-implicit scheme~\reff{ImplicitScheme}, as expected,  is first-order and second-order accurate in time and spatial discretization, respectively. Note that the mesh ratio  chosen here is for the purpose of numerical accuracy test, not for the purpose of the stability or positivity.

To demonstrate the advantage of the proposed Newton's iteration method in computational efficiency, we compare the Newton's iteration method with a simple fixed-point iteration method for solving the nonlinear discrete system~\eqref{disPNP} in terms of computational time and iteration steps. We solve the problem using various mesh resolution with a mesh ratio $\Delta t=h/10$. As shown in Table~\ref{t:NewVsFix}, the number of iteration steps of the Newton's method is about $2$ in each time step evolution, which is significantly fewer than that of the fixed-point method. Also, the fixed-point method takes roughly twofold to threefold longer computational time. 
The computational advantage makes the proposed Newton's iteration method promising in studying complex ion transport problems.
\begin{table}[H]
\centering
\begin{tabular}{ccccccc}
\hline  \hline
 \multirow{2}{*}{Mesh Size} &\multicolumn{2}{c}{Newton's method} &\multicolumn{2}{c}{Fixed-point method} \\
 \cline{2-3}  \cline{4-5}  \cline{6-7}
 & Computational time & Iteration steps & Computational time & Iteration steps \\
 \hline
$50^2$ & 82s & 2& 194s & 21\\
$100^2$ & 1243s & 2& 3615s & 21\\
$150^2$ & 6892s & 2& 20274s & 20\\
$200^2$ & 24798s & 2& 73357s & 19\\
$250^2$ & 68512s & 2& 191746s & 19\\
$300^2$ & 155486s & 2& 404593s & 19\\
 \hline  \hline
\end{tabular}
\caption{Computational time and iteration steps of the proposed Newton's iteration method and fixed-point iteration method in each time step evolution up to time $T=0.1$.}\label{t:NewVsFix}
\end{table}

\subsection{Conservation and energy dissipation}
In this case, we consider a closed, neutral system that consists of symmetric monovalent ions with the following initial and boundary conditions
\begin{equation}\label{InitBouCons}
\left\{
\begin{aligned}
&\psi(t,0,y)=0,~\psi(t,1,y)=1, \quad y\in[0,1],\\
&\frac{\partial\psi}{\partial y}(t,x,0)=\sin\left(\pi x\right),~\frac{\partial\psi}{\partial y}(t,x,1)=-\sin\left( \pi x\right),~~x\in [0,1],\\
&c^1(0,x,y)=1,~ c^2(0,x,y)=1, \quad (x,y)\in [0,1]\times[0,1],  \\
 &\frac{\partial c^1}{\partial  \textbf{n}}+c^1 \frac{\partial \psi}{\partial  \textbf{n}} =0\quad \mbox{and}\quad \frac{\partial c^2}{\partial  \textbf{n}}-c^2 \frac{\partial \psi}{\partial  \textbf{n}} =0\quad \text{on } \partial\Omega.
\end{aligned}
\right.
\end{equation}
The prescribed potential boundary conditions represent that a potential difference is applied horizontally and the upper and lower boundaries carry surface charges with opposite signs.  With such zero-flux boundary conditions and time-independent boundary potentials, the system has properties of mass conservation and free-energy dissipation. 

\begin{figure}[H]
\centering
\includegraphics[scale=.55]{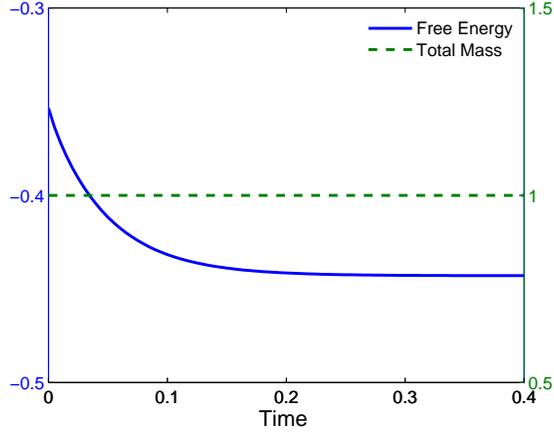}
\caption{ Profiles of the free energy (solid line) and total ion concentrations (dotted line) against time evolution.}
\label{f:MassEnergy}
\end{figure}
As displayed in Fig.~\ref{f:MassEnergy}, our numerical method perfectly conserves the total ion concentration, and the discrete free energy \reff{Fh} decays monotonically and robustly. In addition, our numerical solutions of concentrations are all positive, being consistent with our analysis on solution positivity.

\subsection{Adaptive time stepping}
\begin{figure}[H]
\centering
\includegraphics[scale=.55]{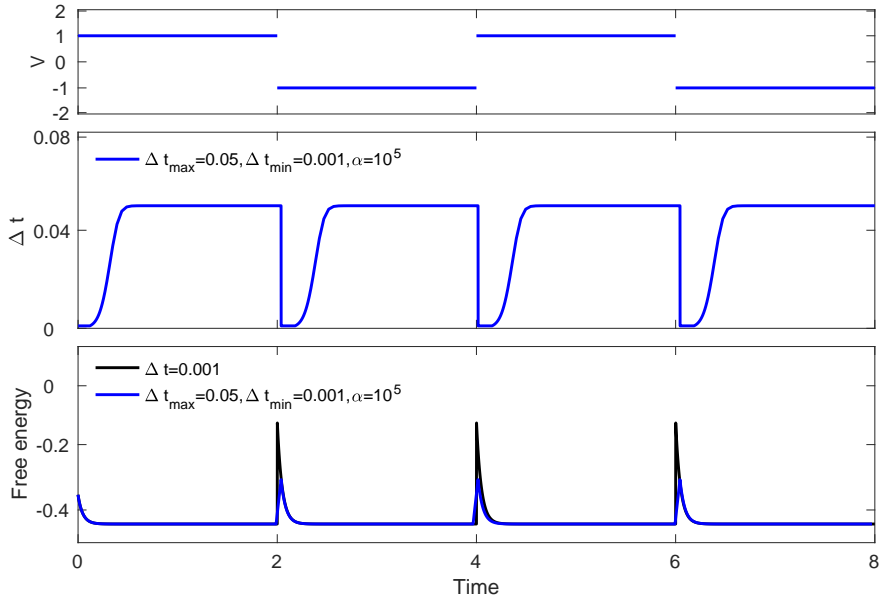}
\caption{Upper: The evolution of the applied potential $V(t)$; Middle: The adaptive time stepping size versus time; Lower: Free-energy evolution with a uniform time-stepping size $\Delta t=0.001$ (black) and adaptive time stepping sizes with parameters $\Delta t_{\text{max}}=0.05$, $\Delta t_{\text{min}}=0.001$, and $\alpha=10^5$ (blue). }
 \label{f:TimeStep}
\end{figure}

We now study the time-marching stability of the numerical scheme and the efficiency improvement by using the strategy of adaptive time stepping in solving problems in which electrolytes are exposed to sudden alternating applied potentials. 
We consider the same initial and boundary conditions as in~\reff{InitBouCons} except that $\psi(t,1,y)=V(t)$ for $y\in[0,1]$ with
\[
V(t)=\chi_{[0,2)}-\chi_{[2,4)}+\chi_{[4,6)}-\chi_{[6,8)},
\]
where $\chi_{[\cdot, \cdot)}$ is the characteristic function of a time interval. That is, a periodically alternating potential $V(t)$ is applied horizontally across the computational domain; cf.~the upper plot of the Fig.~\ref{f:TimeStep}.

We investigate the effectiveness of adaptive time-stepping techniques using the form of \reff{AdaptiveT} with a parameter setting $\Delta t_{\text{max}}=0.05$, $\Delta t_{\text{min}}=0.001$, and $\alpha=10^5$.  The adaptive time step evolution is presented in the middle plot of the Fig.~\ref{f:TimeStep}, and the corresponding free-energy evolution profile is displayed in the lower plot of the Fig.~\ref{f:TimeStep}. For comparison, we also present the free-energy evolution profile computed with a uniform time step size ($\Delta t=0.001$) in the same plot. As the applied potential changes periodically, the free energy correspondingly undergoes large, abrupt changes. This in turn leads to drastic decrease of the time stepping size to the minimum value $\Delta t_{\text{min}}$. One can observe that, for a fixed applied potential, the free energy quickly relaxes and the corresponding time step size increases to its maximum value $\Delta t_{\text{max}}$.    
In contrast to $8000$ steps with $\Delta t=\Delta t_{\text{min}}=0.001$ to reach $T=8$, the adaptive time stepping only takes a total of $877$ steps, which have about $89\%$ reduction in time steps. Of interest is that the free-energy evolution profile computed with adaptive time stepping is almost identical to that with $\Delta t=0.001$, with minor discrepancy due to the resolution of time. This indicates that the strategy of adaptive time stepping is useful in speeding up computations of problems that have applied potentials with sudden changes. 

\subsection{Charge Dynamics}
\begin{figure}[H]
\centering
\includegraphics[scale=.45]{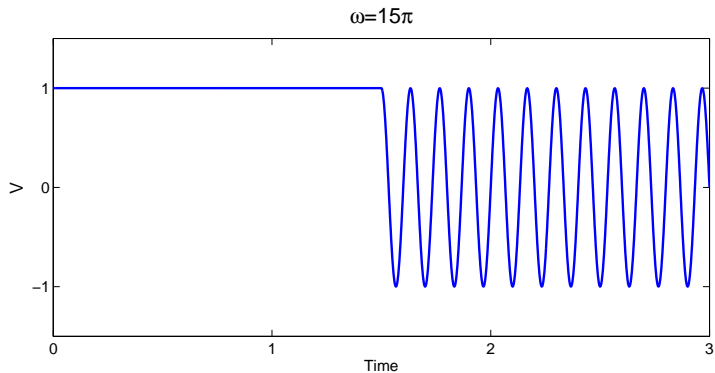}
\caption{The applied potential $V(t)$ with $\omega=15\pi$. It keeps as a constant for the time interval $[0, 1.5)$ and continuously becomes a sinusoidal function for $[1.5, +\infty)$. }
 \label{f:AppliedV}
\end{figure}
We apply the proposed numerical approaches to probe the charge dynamics in electrolytes between two parallel electrodes with sinusoidal applied potentials. Electrolytes under alternating current (AC) have wide range of applications, including AC electroosmosis (ACEO) pumps, cyclic voltammetry, and dielectrophoresis~\cite{BTA:PRE:04, BazantReview_ACIS09, AmeriMiller_PRL18}. Numerical simulation with sinusoidal applied potentials of large frequency is computationally challenging. It sets high demand on the stability of numerical integration methods. 

We consider a closed, neutral system that consists of binary monovalent ions. Due to geometry symmetry, the problem can be reduced to one dimension on a computational domain $[-1, 1]$, after nondimensionalization.  We consider the same initial and boundary conditions as in~\reff{InitBouCons} except that the left electrode is kept grounded, i.e., $\psi(t,-1)=0$, and $\psi(t,1)=V(t)$ with
\[
V(t)=\chi_{[0, 1.5)}+\chi_{[1.5, +\infty)}\sin\left(\omega t\right),
\]
where $\omega$ is the angular frequency; cf.~Fig.~\ref{f:AppliedV}. 
\begin{figure}[H]
\centering
\includegraphics[scale=.55]{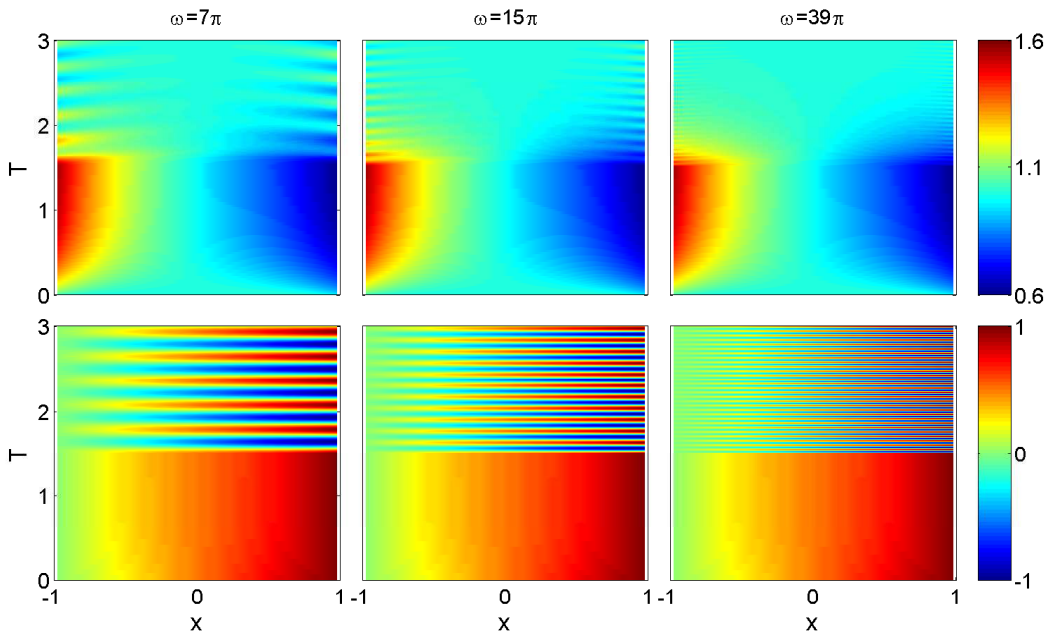}
\caption{Evolution of cation distribution (upper row) and electrostatic potential (lower row) with various angular frequency $\omega$.}
\label{f:ChargeDyms}
\end{figure}

In Fig.~\ref{f:ChargeDyms}, we consider the dynamics of concentrations and electrostatic potential with various angular frequency. For $T<1.5$, the potential difference across two electrodes attracts oppositely charged ions from the bulk, forming electric double layers (EDLs) in the vicinity of electrodes. In the meantime, the electrostatic potential gets screened by the charges in the EDLs.  After the charging phase, the boundary potential becomes sinusoidal and the electrostatic potential across the system responds instantaneously. The ionic concentration close to the electrodes oscillates correspondingly. However, the magnitude of oscillation decays as time evolves, because charges in the EDLs are gradually released to the bulk. 

With larger angular frequency, the electrostatic potential inside still can follow the boundary potential instantaneously. However, the ionic concentration fall behind the potential oscillation and the potential only partially gets screened. At lower frequency, ions can travel much longer distances during each AC cycle; therefore, the oscillatory effect extend farther away from the electrode surface, with more pronounced impact on the structure of EDLs.  For $\omega=39\pi$, we barely observe oscillations in concentration for $2<T<3$. 

\begin{figure}[H]
\centering
\includegraphics[scale=.65]{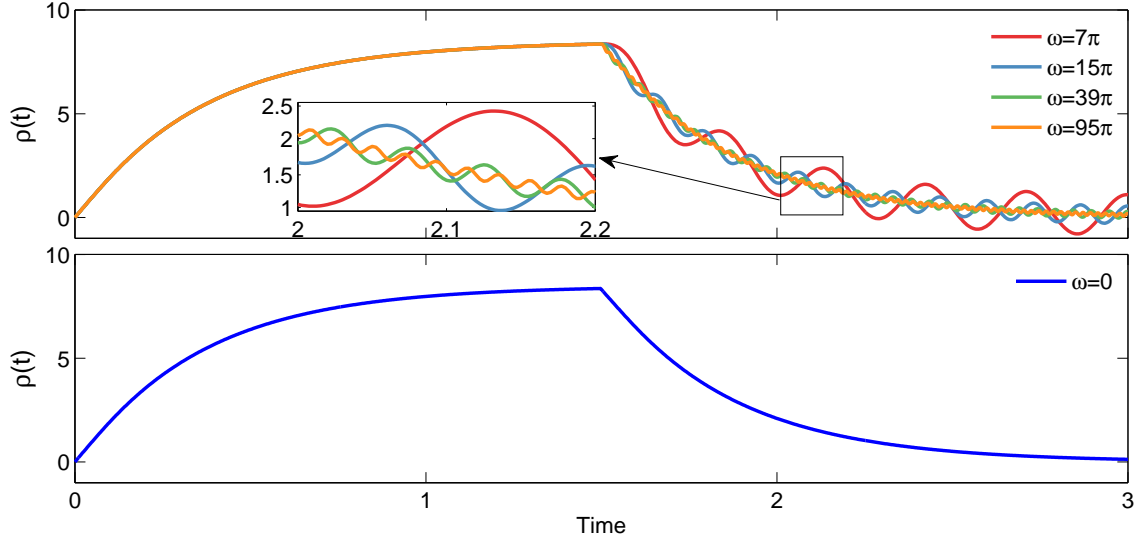}
\caption{Evolution of total net charges $\rho(t)$ with various angular frequency $\omega$.}
 \label{f:rhot}
\end{figure}
To further understand charge dynamics with sinusoidal applied potentials, we also study the evolution of total net charges in left half of the electrolytes~\cite{BTA:PRE:04}:
$$\rho(t)=\int_{-1}^0 \sum_{l=1}^2 q^l c^l(x)dx.$$
As seen from Fig.~\ref{f:rhot}, the total net charges increase quickly and reach a plateau with a constant applied voltage in the charging phase, i.e., $0<T<1.5$. Later, the total net charges decay oscillatorily to zero at $T=3$, which indicates that the discharging phase with sinusoidal applied potentials takes roughly the same time as the charging phase. For lower angular frequency, the total net charges have larger magnitude of oscillation, being consistent with the results shown in Fig.~\ref{f:ChargeDyms}. For $\omega=95\pi$, the oscillation in the profile of $\rho(t)$ is extremely small. Larger angular frequency leads to shrinking oscillation magnitude, indicating that the effect of applied potentials with rather large frequency is equivalent to a zero applied potential. This observation is further confirmed by the lower plot in the Fig.~\ref{f:rhot}, which is obtained with a zero applied potential ($\omega=0$) for $T>1.5$. We can see that the profile of the total net charges is almost identical to the case with $\omega=95\pi$.

\section{Conclusions}\label{s:Conclusions}
The Poisson--Nernst--Planck (PNP) equations are a classical model to describe ion transport, which is fundamental to many applications.  In this work, we have developed finite difference schemes for solving the multi-dimensional PNP equations with multiple ionic species. Numerical analysis has shown that the schemes are able to guarantee mass conservation, positivity, and free-energy dissipation at fully discrete level.  The novelty of numerical schemes lies in using the harmonic-mean approximations in the spatial discretization of the Nernst--Planck equations. In addition, we have proposed a new Newton's method to efficiently solve a fully implicit nonlinear system resulting from discretization. The improved computational efficiency of the Newton's method originates from the usage of the electrostatic potential as the iteration variable, instead of both the potential and concentration of multiple ionic species.  Thanks to the harmonic-mean approximations, we have been able to rigorously establish the solvability and stability of the linearized problem in the Newton's method and estimates on the upper bound of condition numbers of coefficient matrices in linear systems that are solved iteratively. Extensive numerical tests have been performed to corroborate the anticipated numerical accuracy, computational efficiency, and structure-preserving properties of the developed schemes. Also, numerical simulations have shown that a strategy of adaptive time stepping is able to speed up simulation of problems with time-dependent, alternating boundary potentials. Finally, the proposed numerical approaches have been applied to understand ion transport in response to a sinusoidal applied potential. Such numerical simulations demonstrate that the developed numerical approaches are promising in solving complex, realistic ion transport problems.

We now discuss several issues and possible further refinements of our work. The current development of numerical methods only considers a regular domain. However, it is of practical interest to extend the developed numerical methods for irregular computational domains, such as irregular geometry considered in ion-channel applications. The discretization accuracy on irregular boundaries deserves further attention, and discrete structure-preserving properties, such as positivity preservation and mass conservation, are highly desirable to maintain. Because these numerical properties are crucial to the validity of numerical solutions. It is also desirable to pursue second-order temporal discretization schemes that have structure-preserving properties.  In addition, it will be our future work to extend the proposed discretization schemes and the Newton's method to solve other modified PNP models that account for steric effects, Coulomb ionic correlations, and inhomogeneous dielectric effects~\cite{Bazant_PRE07I, BazantReview_ACIS09,  HyonLiuBob_CMS10, BZLu_BiophyJ11, HyonLiuBob_JPCB12, LiuJiXu_SIAP18, GavishLiuEisenberg_JPCB18}.

\vskip 5mm

\noindent{\bf Acknowledgments.}
J. Ding was supported by Postgraduate Research~\& Practice Innovation Program of Jiangsu Province and National Natural Science Foundation of China (No. 21773165, 11601361,  11771318, and 11790274).  S. Zhou was supported by National Natural Science Foundation of China (No. 21773165 and 11601361) and National Key R\&D Program of China  (No. 2018YFB0204404).

\appendix
\numberwithin{equation}{section}
\makeatletter
\newcommand{\section@cntformat}{Appendix \thesection:\ }
\makeatother

\section{Appendix: the matrix $\mathcal{K}(\pmb{\mu}^{l},\bu)$}
\label{ApxA}
\setcounter{equation}{0}
To facilitate the presentation, we denote by $S^l=q^l\bu$ and introduce the following discrete operators
$$\mathcal{E}^{\pm}_x S^{l,n}_{i,j}=\frac{e^{S^{l,n}_{i\pm 1,j}-S^{l,n}_{i,j}}}{(1+e^{S^{l,n}_{i\pm 1,j}-S^{l,n}_{i,j}})^2},~ \mathcal{E}^{\pm}_y S^{l,n}_{i,j}=\frac{e^{S^{l,n}_{i,j\pm 1}-S^{l,n}_{i,j}}}{(1+e^{S^{l,n}_{i,j\pm 1}-S^{l,n}_{i,j}})^2}.$$
We now examine elements in each column of the matrix $\mathcal{K}$. For the $k$th column with $k=[i,j]$,  the column elements are related to  coefficients in the discretization stencils of the Nernst--Planck equations, associated to the grid point $\{x_i, y_j\}$. Thus, we look into all of the discretization stencils for different types of grid points.

First, for any interior grid point $\{x_i, y_j\}$ with $i=2, \dots , N_x-1$ and $j=2, \dots , N_y-1$,  non-zero entries of the $k$th column ($k=[i,j]$) are given by
\begin{equation*}
\mathcal{K}^l_{m,k}=\left\{
\begin{aligned}
& -2\Delta t^{n+1} \frac{h^y_j}{h^x_{i-\frac{1}{2}}}\mathcal{E}^{-}_x S^{l,n}_{i,j}\mu^{l,n}_{m+N_y}-2\Delta t^{n+1}\frac{h^y_j}{h^x_{i-\frac{1}{2}}}\mathcal{E}^{+}_x S^{l,n}_{i-1,j}\mu^{l,n}_m ,\quad m=[i-1,j],\\
& -2\Delta t^{n+1} \frac{h^x_i}{h^y_{j-\frac{1}{2}}}\mathcal{E}^{+}_y S^{l,n}_{i,j-1}\mu^{l,n}_m-2\Delta t^{n+1} \frac{h^x_i}{h^y_{j-\frac{1}{2}}}\mathcal{E}^{-}_y S^{l,n}_{i,j} \mu^{l,n}_{m+1},\quad m=[i,j-1],\\
& 2\Delta t^{n+1} \left(\frac{h^y_j}{h^x_{i+\frac{1}{2}}}  \mathcal{E}^{+}_x S^{l,n}_{i,j}+  \frac{h^y_j}{h^x_{i-\frac{1}{2}}}\mathcal{E}^{-}_x S^{l,n}_{i,j}
 +\frac{h^x_i}{h^y_{j+\frac{1}{2}}}\mathcal{E}^{+}_y S^{l,n}_{i,j}+ \frac{h^x_i}{h^y_{j-\frac{1}{2}}}\mathcal{E}^{-}_y S^{l,n}_{i,j} \right)\mu^{l,n}_m\\& \quad + 2\Delta t^{n+1}\frac{h^y_j}{h^x_{i-\frac{1}{2}}}\mathcal{E}^{+}_x S^{l,n}_{i-1,j}\mu^{l,n}_{m-N_y}+2\Delta t^{n+1}\frac{h^y_j}{h^x_{i+\frac{1}{2}}}\mathcal{E}^{-}_x S^{l,n}_{i+1,j}\mu^{l,n}_{m+N_y}\\& \quad +2\Delta t^{n+1}\frac{h^x_i}{h^y_{j+\frac{1}{2}}}\mathcal{E}^{-}_y S^{l,n}_{i,j+1}\mu^{l,n}_{m+1}+2\Delta t^{n+1} \frac{h^x_i}{h^y_{j-\frac{1}{2}}}\mathcal{E}^{+}_y S^{l,n}_{i,j-1}\mu^{l,n}_{m-1}, ~m=k,\\
&-2\Delta t^{n+1}\frac{h^x_i}{h^y_{j+\frac{1}{2}}}\mathcal{E}^{-}_y S^{l,n}_{i,j+1}\mu^{l,n}_{m+1}-2\Delta t^{n+1} \frac{h^x_i}{h^y_{j+\frac{1}{2}}}\mathcal{E}^{+}_y S^{l,n}_{i,j}\mu^{l,n}_{m},\quad m=[i,j+1],\\
&-2\Delta t^{n+1}\frac{h^y_j}{h^x_{i+\frac{1}{2}}}\mathcal{E}^{-}_x S^{l,n}_{i+1,j}\mu^{l,n}_{m+N_y}-2\Delta t^{n+1}\frac{h^y_j}{h^x_{i+\frac{1}{2}}}  \mathcal{E}^{+}_x S^{l,n}_{i,j}\mu^{l,n}_{m},\quad m=[i+1,j].\\
\end{aligned}
\right.
\end{equation*}
Hence, the conclusion \reff{Bpro} holds for the columns associated to interior grid points. We examine the conclusion for boundary grid points that are adjacent to four edges but not corner vertices.  We consider Dirichlet boundary grid points $\{x_1, y_j\}$ for $j=2, \dots , N_y-1$, and Neumann boundary grid points $\{x_i, y_1\}$ for $i=2, \dots , N_x-1$. Non-zero entries of the $k$th column ($k=[1,j]$) are
\begin{equation*}
\mathcal{K}^l_{m,k}=\left\{
\begin{aligned}
& -2\Delta t^{n+1} \frac{h^x_1}{h^y_{j-\frac{1}{2}}}\mathcal{E}^{+}_y S^{l,n}_{1,j-1}\mu^{l,n}_{m}-2\Delta t^{n+1} \frac{h^x_1}{h^y_{j-\frac{1}{2}}}\mathcal{E}^{-}_y S^{l,n}_{1,j} \mu^{l,n}_{m+1},\quad m=[1,j-1],\\
& 2\Delta t^{n+1} \left(\frac{h^y_j}{h^x_{\frac{3}{2}}}  \mathcal{E}^{+}_x S^{l,n}_{1,j}
 +\frac{h^x_1}{h^y_{j+\frac{1}{2}}}\mathcal{E}^{+}_y S^{l,n}_{1,j}+ \frac{h^x_1}{h^y_{j-\frac{1}{2}}}\mathcal{E}^{-}_y S^{l,n}_{1,j} \right)\mu^{l,n}_{m}+2\Delta t^{n+1}\frac{h^y_j}{h^x_{\frac{3}{2}}}\mathcal{E}^{-}_x S^{l,n}_{2,j}\mu^{l,n}_{m+N_y} \\
 & \quad +2\Delta t^{n+1}\frac{h^x_1}{h^y_{j+\frac{1}{2}}}\mathcal{E}^{-}_y S^{l,n}_{1,j+1}\mu^{l,n}_{m+1}+2\Delta t^{n+1} \frac{h^x_1}{h^y_{j-\frac{1}{2}}}\mathcal{E}^{+}_y S^{l,n}_{1,j-1}\mu^{l,n}_{m-1}, ~m=k,\\
&-2\Delta t^{n+1}\frac{h^x_1}{h^y_{j+\frac{1}{2}}}\mathcal{E}^{-}_y S^{l,n}_{1,j+1}\mu^{l,n}_{m}-2\Delta t^{n+1} \frac{h^x_1}{h^y_{j+\frac{1}{2}}}\mathcal{E}^{+}_y S^{l,n}_{1,j}\mu^{l,n}_{m-1},\quad m=[1,j+1],\\
&-2\Delta t^{n+1}\frac{h^y_j}{h^x_{\frac{3}{2}}}\mathcal{E}^{-}_x S^{l,n}_{2,j}\mu^{l,n}_{m}-2\Delta t^{n+1}\frac{h^y_j}{h^x_{\frac{3}{2}}}  \mathcal{E}^{+}_x S^{l,n}_{1,j}\mu^{l,n}_{m-N_y},\quad m=[2,j].\\
\end{aligned}
\right.
\end{equation*}
Non-zero entries of the $k$th column ($k=[i,1]$) are
\begin{equation*}
\mathcal{K}^l_{m,k}=\left\{
\begin{aligned}
& -2\Delta t^{n+1} \frac{h^y_1}{h^x_{i-\frac{1}{2}}}\mathcal{E}^{-}_x S^{l,n}_{i,1}\mu^{l,n}_{m+N_y}-2\Delta t^{n+1}\frac{h^y_1}{h^x_{i-\frac{1}{2}}}\mathcal{E}^{+}_x S^{l,n}_{i-1,1}\mu^{l,n}_{m} ,\quad m=[i-1,1],\\
& 2\Delta t^{n+1} \left(\frac{h^y_1}{h^x_{i+\frac{1}{2}}}  \mathcal{E}^{+}_x S^{l,n}_{i,1}+  \frac{h^y_1}{h^x_{i-\frac{1}{2}}}\mathcal{E}^{-}_x S^{l,n}_{i,1}
 +\frac{h^x_i}{h^y_{\frac{3}{2}}}\mathcal{E}^{+}_y S^{l,n}_{i,1} \right)\mu^{l,n}_{m}+ 2\Delta t^{n+1}\frac{h^y_1}{h^x_{i-\frac{1}{2}}}\mathcal{E}^{+}_x S^{l,n}_{i-1,1}\mu^{l,n}_{m-N_y}\\
&\quad +2\Delta t^{n+1}\frac{h^y_1}{h^x_{i+\frac{1}{2}}}\mathcal{E}^{-}_x S^{l,n}_{i+1,1}\mu^{l,n}_{m+N_y}+2\Delta t^{n+1}\frac{h^x_i}{h^y_{\frac{3}{2}}}\mathcal{E}^{-}_y S^{l,n}_{i,2}\mu^{l,n}_{m+1}, ~m=k,\\
&-2\Delta t^{n+1}\frac{h^x_i}{h^y_{\frac{3}{2}}}\mathcal{E}^{-}_y S^{l,n}_{i,2}\mu^{l,n}_{m}-2\Delta t^{n+1} \frac{h^x_i}{h^y_{\frac{3}{2}}}\mathcal{E}^{+}_y S^{l,n}_{i,1}\mu^{l,n}_{m-1},\quad m=[i,2],\\
&-2\Delta t^{n+1}\frac{h^y_1}{h^x_{i+\frac{1}{2}}}\mathcal{E}^{-}_x S^{l,n}_{i+1,1}\mu^{l,n}_{m}-2\Delta t^{n+1}\frac{h^y_1}{h^x_{i+\frac{1}{2}}}  \mathcal{E}^{+}_x S^{l,n}_{i,1}\mu^{l,n}_{m-N_y},\quad m=[i+1,1].\\
\end{aligned}
\right.
\end{equation*}
Thus, we can verify that the conclusion \reff{Bpro} holds for the columns associated to edge grid points. For corner vertices grid points, e.g., $k=[1,1]$, we have
\begin{equation*}
\mathcal{K}^l_{m,k}=\left\{
\begin{aligned}
& 2\Delta t^{n+1} \left(\frac{h^y_1}{h^x_{\frac{3}{2}}}  \mathcal{E}^{+}_x S^{l,n}_{1,1}
 +\frac{h^x_1}{h^y_{\frac{3}{2}}}\mathcal{E}^{+}_y S^{l,n}_{1,1} \right)\mu^{l,n}_{m}+2\Delta t^{n+1}\frac{h^y_1}{h^x_{\frac{3}{2}}}\mathcal{E}^{-}_x S^{l,n}_{2,1}\mu^{l,n}_{m+N_y}\\
 & \quad +2\Delta t^{n+1}\frac{h^x_1}{h^y_{\frac{3}{2}}}\mathcal{E}^{-}_y S^{l,n}_{1,2}\mu^{l,n}_{m+1}, ~m=k,\\
&-2\Delta t^{n+1}\frac{h^x_1}{h^y_{\frac{3}{2}}}\mathcal{E}^{-}_y S^{l,n}_{1,2}\mu^{l,n}_{m}-2\Delta t^{n+1} \frac{h^x_1}{h^y_{\frac{3}{2}}}\mathcal{E}^{+}_y S^{l,n}_{1,1}\mu^{l,n}_{m-1},\quad m=[1,2],\\
&-2\Delta t^{n+1}\frac{h^y_1}{h^x_{\frac{3}{2}}}\mathcal{E}^{-}_x S^{l,n}_{2,1}\mu^{l,n}_{m}-2\Delta t^{n+1}\frac{h^y_1}{h^x_{\frac{3}{2}}}  \mathcal{E}^{+}_x S^{l,n}_{1,1}\mu^{l,n}_{m-N_y},\quad m=[2,1].\\
\end{aligned}
\right.
\end{equation*}
Other corner vertices can be verified analogously. Also, we can verify that $\mathcal{K}$ is a symmetric matrix. 

Without any assumption on $S^{l,n}_{i,j}$, we have
\begin{equation}\label{EBound}
0<\mathcal{E}^{\pm}_{k} S^{l,n}_{i,j}\leq \frac{1}{4} ~~\mbox{for}~ k=x, y.
\end{equation}
Therefore, we have the following results:
\[
\left\{
\begin{aligned}
&\sum_{m=1}^{N_xN_y} \mathcal{K}^l_{m,n} =\sum_{n=1}^{N_xN_y} \mathcal{K}^l_{m,n} =0 \quad \mbox{for}~ m,n=1, \dots, N_xN_y,\\
&0<\mathcal{K}^l_{n,n} <2\Delta t^{n+1}(\frac{h^y_M}{h^x_m}+\frac{h^x_M}{h^y_m})\|\pmb{\mu}^l\|_{\infty} \quad \mbox{for}~ n=1, \dots, N_xN_y,\\
&-\Delta t^{n+1}\max\{\frac{h^y_M}{h^x_m},\frac{h^x_M}{h^y_m}\}\|\pmb{\mu}^l\|_{\infty}<\mathcal{K}^l_{m,n} <0 \quad \mbox{for}~ n,m=1, \dots, N_xN_y,~ \mbox{and}~ m\neq n.
\end{aligned}
\right.
\]
\bibliographystyle{plain}
\bibliography{PNP}

\end{document}